\documentclass[10pt,letterpaper,twoside,reqno]{amsart}

\makeatletter{}
\usepackage{fixltx2e}                       
\usepackage[usenames,dvipsnames]{xcolor}
\usepackage{fancyhdr}
\usepackage{amsmath,amsfonts,amsbsy,amsgen,amscd,mathrsfs,amssymb,amsthm}
\usepackage{subfig}
\usepackage{url}

\usepackage{mathtools}

\mathtoolsset{showonlyrefs}

\usepackage[font=small,margin=0.25in,labelfont={sc},labelsep={colon}]{caption}

\usepackage{tikz}
\usepackage{microtype}
\usepackage{enumitem}

\definecolor{dark-gray}{gray}{0.3}
\definecolor{dkgray}{rgb}{.4,.4,.4}
\definecolor{dkblue}{rgb}{0,0,.5}
\definecolor{medblue}{rgb}{0,0,.75}
\definecolor{rust}{rgb}{0.5,0.1,0.1}

\usepackage[colorlinks=true]{hyperref}

\usepackage{setspace}

\usepackage{graphicx}
\usepackage{booktabs,longtable,tabu} 
\usepackage{multirow} 
\usepackage{float}

\usepackage[T1]{fontenc}

\usepackage{fourier}
\usepackage{charter}

\usepackage{bm} 
\graphicspath{{figures/}}

\newtheorem{theorem}{Theorem}[section]
\newtheorem{lemma}[theorem]{Lemma}

\newtheorem{proposition}[theorem]{Proposition}
\newtheorem{fact}[theorem]{Fact}

\newtheorem{corollary}[theorem]{Corollary}

\theoremstyle{definition}

\newtheorem{definition}[theorem]{Definition}

\newtheorem{remark}[theorem]{Remark}

\newcommand{\term}{\emph}

\numberwithin{equation}{section} 
\numberwithin{figure}{section}
\numberwithin{table}{section}

\floatstyle{plaintop}
\newfloat{recipe}{thp}{lor}
\floatname{recipe}{Recipe}
\numberwithin{recipe}{section}

\providecommand{\mathbold}[1]{\bm{#1}}                                                                  
\newcommand{\qtq}[1]{\quad\text{#1}\quad}

\renewcommand{\phi}{\varphi}

\newcommand{\eps}{\varepsilon}

\newcommand{\half}{\tfrac{1}{2}}

\newcommand{\econst}{\mathrm{e}}

\newcommand{\coll}[1]{\mathscr{#1}}
\newcommand{\sphere}[1]{\mathsf{S}^{#1}}

\providecommand{\mathbbm}{\mathbb} 
\newcommand{\R}{\mathbbm{R}}

\newcommand{\polar}{\circ}

\newcommand{\abs}[1]{\left\vert {#1} \right\vert}

\newcommand{\diff}[1]{\mathrm{d}{#1}}
\newcommand{\idiff}[1]{\, \diff{#1}}

\newcommand{\grad}{\nabla} 

\newcommand{\argmin}{\operatorname*{arg\; min}}

\newcommand{\Prob}{\mathbbm{P}}

\newcommand{\Expect}{\operatorname{\mathbb{E}}}

\DeclareMathOperator{\Var}{Var}

\newcommand{\vct}[1]{\mathbold{#1}}
\newcommand{\mtx}[1]{\mathbold{#1}}

\newcommand{\transp}{T}

\newcommand{\trace}{\operatorname{tr}}

\newcommand{\Proj}{\ensuremath{\mtx{\Pi}}}

\newcommand{\ip}[2]{\left\langle {#1},\ {#2} \right\rangle}

\newcommand{\norm}[1]{\left\Vert {#1} \right\Vert}

\DeclareMathOperator{\dist}{dist}

\newcommand{\enorm}[1]{\norm{#1}}
\newcommand{\enormsm}[1]{\enorm{\smash{#1}}}

\newcommand{\enormsq}[1]{\enorm{#1}^2}

\newcommand{\fnorm}[1]{\norm{#1}_{\mathrm{F}}}

 \newcommand{\sdim}{\delta}

\DeclareMathOperator{\Circ}{Circ}

\newcommand{\cC}{\mathcal{C}} 
 
\newcommand{\sang}{\angle}

\newcommand{\relint}{\operatorname{relint}}
\newcommand{\cone}{\operatorname{cone}}
\newcommand{\lin}{\operatorname{lin}}

\DeclareMathOperator{\aTube}{\mathcal{T}_{s}}
\DeclareMathOperator{\distH}{\dist_{\mathcal{H}}}

\newcommand{\pInd}{\mathbb{1}}

\renewcommand{\zeta}{\eta}

\hypersetup{pdftitle={From Steiner formulas for cones to concentration of intrinsic volumes}}    \hypersetup{pdfauthor={Michael B.~McCoy and Joel A.~Tropp}}     \hypersetup{pdfkeywords=
{concentration inequality} {convex cone} {Gaussian width} {integral geometry} {intrinsic volume} {geometric probability} {statistical dimension} {Steiner formula} {variance bound}}  \hypersetup{linkcolor=dkblue}          \hypersetup{citecolor=rust}        \hypersetup{urlcolor=rust}           

\title[Steiner Formulas for Cones and Intrinsic Volumes]{From Steiner Formulas for Cones \\
to Concentration of Intrinsic Volumes}

\author[M.~B.~McCoy and J.~A.~Tropp]{Michael B.~McCoy and Joel~A.~Tropp}

\date{23 August 2013.  Revised 20 March 2014, 29 April 2014, and 12 July 2015.}


\begin{document}

\begin{abstract}
The intrinsic volumes of a convex cone are geometric functionals that
return basic structural information about the cone.  Recent research has
demonstrated that conic intrinsic volumes are valuable for understanding the
behavior of random convex optimization problems.
This paper develops a systematic technique for studying conic intrinsic volumes
using methods from probability.  At the heart of this approach is a general Steiner
formula for cones.  This result converts questions about the intrinsic
volumes into questions about the projection of a Gaussian random vector
onto the cone, which can then be resolved using tools from Gaussian analysis.
The approach leads to new identities and bounds for the intrinsic volumes of
a cone, including a near-optimal concentration inequality.
\end{abstract}

\maketitle

\section{Introduction}

In the 1840s, Steiner developed a striking decomposition for the volume of
a Euclidean expansion of a polytope in $\R^3$.
The modern statement of Steiner's formula describes an expansion of
a compact convex set $K$ in $\R^d$:
\begin{equation} \label{eqn:euclid-steiner}
\operatorname{Vol}(K + \lambda \mathsf{B}_d) = \sum_{j=0}^d
	\lambda^{d-j} \cdot \operatorname{Vol}(\mathsf{B}_{d-j}) \cdot
	\mathcal{V}_{j}(K)
\quad\text{for $\lambda \geq 0$.}
\end{equation}
The symbol $\mathsf{B}_j$ refers to the Euclidean unit ball in $\R^j$,
and $+$ denotes the Minkowski sum.
In other words, the volume of the expansion is just a
polynomial whose coefficients depend on the set $K$.
The geometric functionals $\mathcal{V}_j$ that appear
in~\eqref{eqn:euclid-steiner} are called
\term{Euclidean intrinsic volumes}~\cite{McM:75}.  Some of these are
familiar, such as the usual volume $\mathcal{V}_d$, the surface area
$2 \, \mathcal{V}_{d-1}$, and the Euler characteristic $\mathcal{V}_0$.
They can all be interpreted as measures of content that are invariant
under rigid motions and isometric embedding~\cite{Sch:93}.

Beginning around 1940, researchers began to develop analogues
of the Steiner formula in spherical geometry~\cite{Hot:39,Wey:39,Her:43,All:48,San:50}.
In their modern form, these results express the size of an angular
expansion of a closed convex cone $C$ in $\R^d$:
\begin{equation} \label{eqn:sphere-steiner}
\operatorname{Vol}\big\{ \vct{x} \in \sphere{d-1} : \dist^2(\vct{x}, C) \leq \lambda \big\}
	= \sum_{j=0}^d \beta_{j,d}(\lambda) \cdot v_j(C)
	\quad\text{for $\lambda \in [0,1]$.}
\end{equation}
We have written $\sphere{d-1}$ for the Euclidean
unit sphere in $\R^d$, and the functions $\beta_{j,d} : [0, 1] \to \R_+$
do not depend on the cone $C$.
The geometric functionals $v_j$ that appear in~\eqref{eqn:sphere-steiner}
are called \term{conic intrinsic volumes}.\footnote{The $j$th conic intrinsic volume $v_j(C)$ corresponds to the $(j-1)$th spherical intrinsic volume $\nu_{j-1}(C\cap \sphere{d-1})$ that appears in the literature.}
These quantities capture fundamental structural information about a convex cone.
They are invariant under rotation; they do not depend on the embedding dimension;
and they arise in many other geometric problems~\cite{SchWei:08}.

The intrinsic volumes of a closed convex cone $C$ in $\R^d$ satisfy
several important identities~\cite[Thm.~6.5.5]{SchWei:08}.  In particular,
the numbers $v_0(C), \dots, v_d(C)$ are nonnegative and sum to one,
so they describe a probability distribution on the set $\{ 0, 1, 2, \dots, d \}$.
Thus, we can define a random variable $V_C$ by the relations
$$
\Prob\big\{ V_C = k \big\} = v_k(C)
\quad\text{for each $k = 0, 1, 2, \dots, d$.}
$$
This construction invites us to use probabilistic methods to study 
the cone $C$.

Recent research~\cite[Thm.~6.1]{ALMT:13} has determined that
the random variable $V_C$ concentrates sharply about its mean value
for every closed convex cone $C$.  In other words, most of the
intrinsic volumes of a cone have negligible size;
see Figure~\ref{fig:circ-cone} for a typical example.
As a consequence of this phenomenon,
a small number of statistics of $V_C$ capture the salient
information about the cone.  For many purposes, we only need to
know the mean, the variance,
and the type of tail decay.
This paper develops a systematic technique for collecting
this kind of information.

\begin{figure}[t]
\vspace{1pc}
\includegraphics[width=0.50\textwidth]{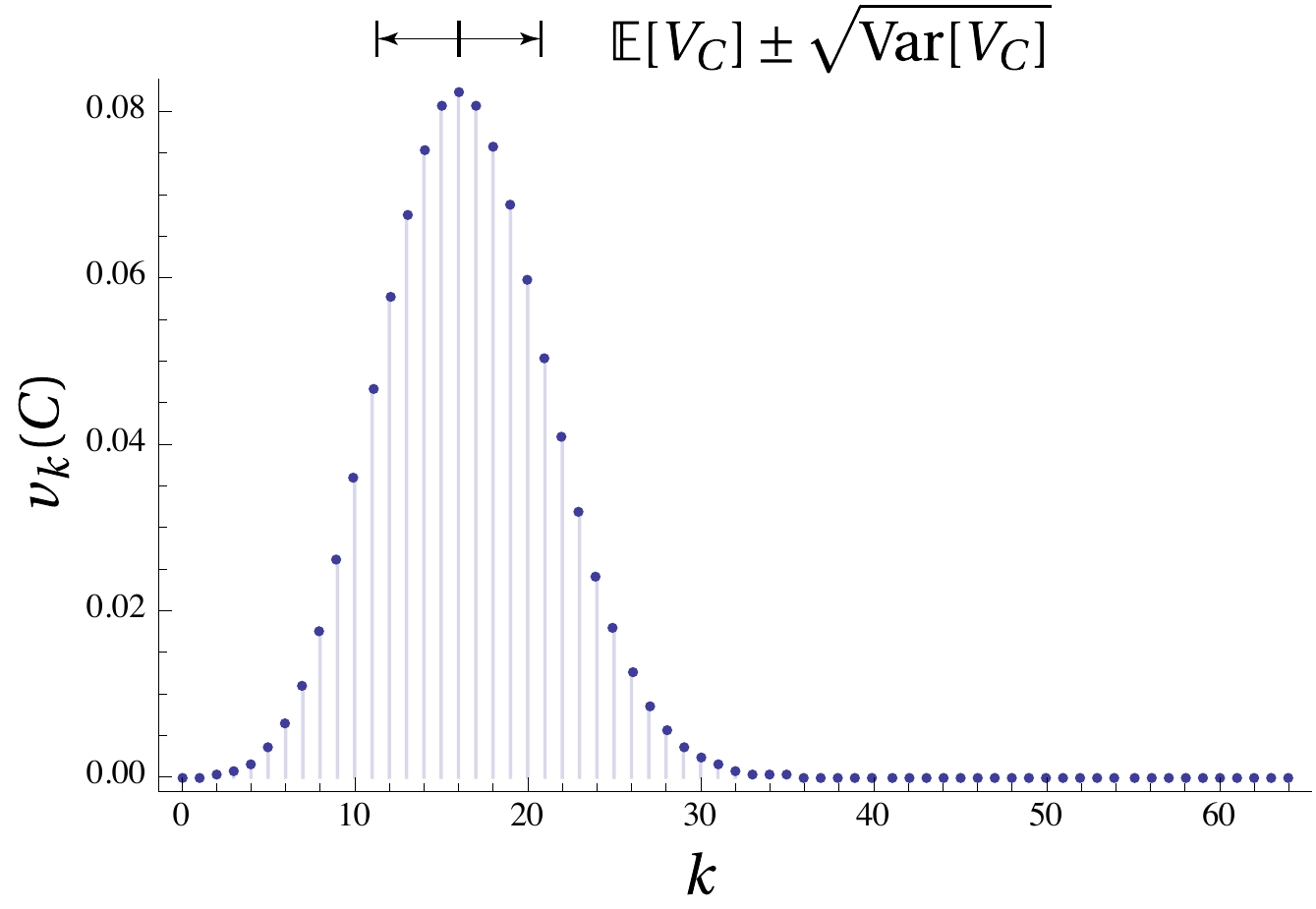}
\hspace{-0.5pc}
\includegraphics[width=0.50\textwidth]{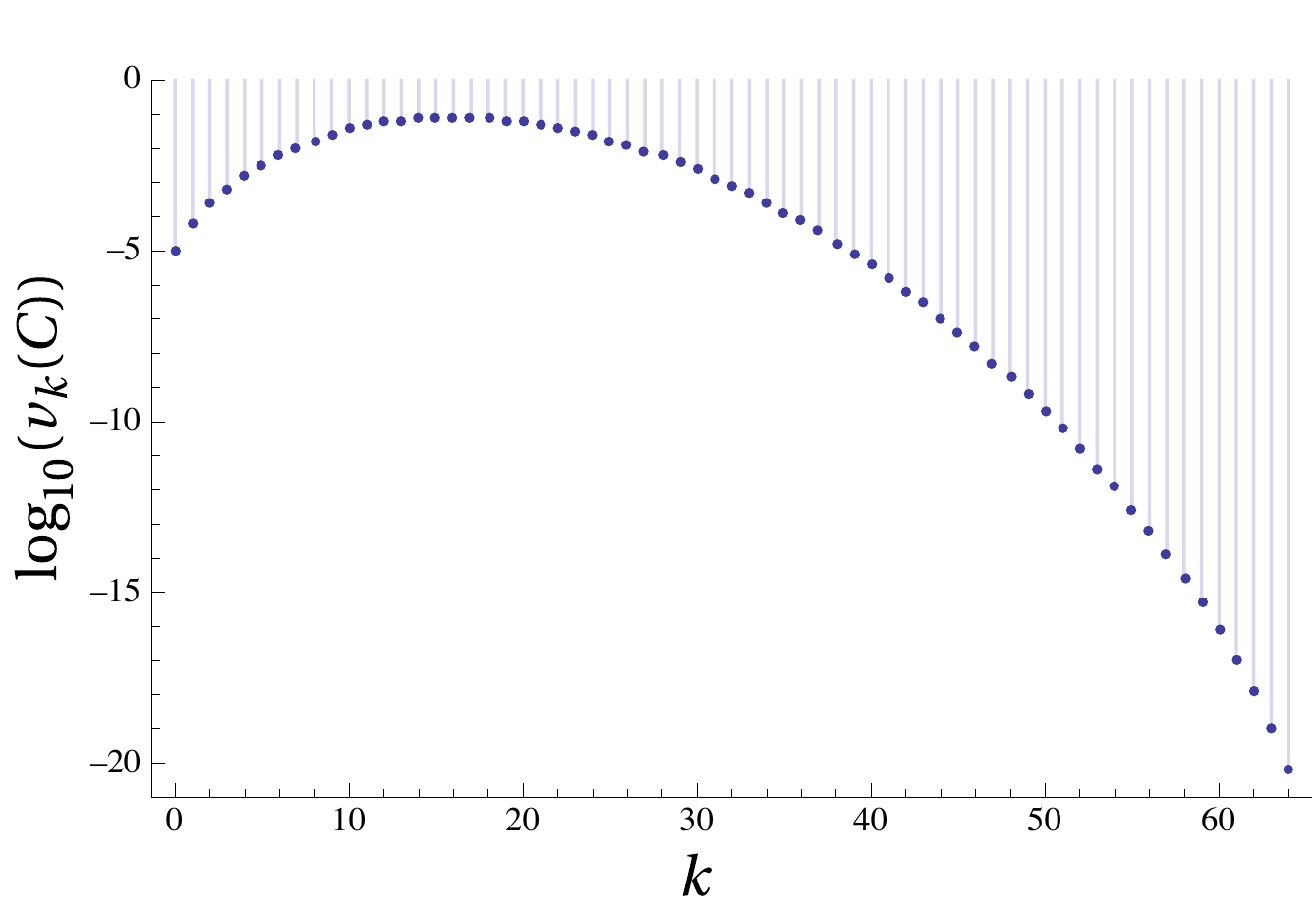}
\caption{\textbf{The intrinsic volumes of a circular cone.}
These two diagrams depict the intrinsic volumes $v_k(C)$
of a circular cone $C$ with angle $\pi/6$ in $\R^{64}$,
computed using the formulas from~\cite[Ex.~4.4.8]{Ame:11}.
The intrinsic volume random variable
$V_C$ has mean $\Expect[V_C] \approx 16.5$ and variance $\Var[V_C] \approx 23.25$.
The tails of $V_C$ exhibit Gaussian decay near the mean and Poisson
decay farther away.  
\textbf{[left]}  The intrinsic volumes $v_k(C)$ coincide with the
probability mass function of $V_C$.
\textbf{[right]} The logarithms $\log_{10}(v_k(C))$ of the intrinsic
volumes have a quadratic profile near $k = \Expect[V_C]$, which is indicative
of the Gaussian decay.}
\label{fig:circ-cone}
\end{figure}

Our method depends on a generalization of the spherical Steiner
formula~\eqref{eqn:sphere-steiner}.
This result, Theorem~\ref{thm:general-steiner},
allows us to compute statistics of the intrinsic volume random variable $V_C$ by
passing to a geometric random variable that we can study
directly.  The prospect for making this type of argument is dimly
visible in the spherical Steiner formula~\eqref{eqn:sphere-steiner}:
the right-hand side can be interpreted as a moment of $V_C$, while the left-hand side
reflects the probability that a certain geometric event occurs.
Our master Steiner formula provides the flexibility we need
to study a wider class of statistics.

This species of argument was developed in collaboration with our
coauthors Dennis Amelunxen and Martin Lotz.
In our earlier joint paper~\cite{ALMT:13},
we used a laborious version of the technique
to show that the intrinsic volumes concentrate.
Here, we simplify and strengthen the method
to obtain new relations for the variance of $V_C$
and to improve the concentration inequalities.

Our work fits into a larger program that studies convex optimization with sophisticated tools from geometry.
The link between conic geometry and convex optimization arises because convex cones play the same role in
convex analysis that subspaces play in linear algebra~\cite[p.~89--90]{HirLem:93}.
Indeed, when we study the behavior of convex optimization problems with random data,
the conic intrinsic volumes arise naturally; see, for example,~\cite{VerSpo:86,VerSpo:92,Don:06a,BurCucLot:08,
DonTan:10,BurCucLot:10, AmeBur:12,AmeBur:14,McCTro:12,AmeBur:13, ALMT:13,McC:13}.
Unfortunately, this program of research has been delayed by the apparent difficulty
of producing explicit bounds for conic intrinsic volumes.
As a consequence, we believe that it is timely to investigate the
concentration properties of the intrinsic volumes of a general cone.

\subsection{Roadmap}

Section~\ref{sec:intvols} contains the definition and basic properties
of the conic intrinsic volumes.
Section~\ref{sec:gen-steiner} states our master Steiner formula
for convex cones, and it explains how to derive formulas for the size of
an expansion of a convex cone.  We begin our probabilistic analysis of
the intrinsic volumes in Section~\ref{sec:prob-intvols}.  This material
includes formulas and bounds for the variance and exponential moments
of the intrinsic volume random variable.  Section~\ref{sec:intvol-product}
continues with a probabilistic treatment of the intrinsic volumes
of a product cone.  We provide several detailed examples of these
methods in Section~\ref{sec:examples}.  Afterward,
Section~\ref{sec:backgr-conv-cones} summarizes some background
material about convex cones in preparation for the proof of
master Steiner formula in Section~\ref{sec:proof-gener-steiner}.

\subsection{Notation and basic concepts}
\label{sec:convex-analysis}

Before commencing with the main development, let us set notation and recall some basic facts from convex analysis.
Section~\ref{sec:backgr-conv-cones} contains a more complete discussion; we provide cross-references as needed. 

We work in the Euclidean space $\R^d$, equipped with the standard inner product $\ip{ \cdot }{ \cdot}$, the associated norm $\norm{ \cdot }$, and the norm topology.  The symbols $\vct{0}$ and $\vct{0}_d$ refer to the origin of $\R^d$.
For a point $\vct{x} \in \R^d$ and a set $K \subset \R^d$, we define the distance $\dist( \vct{x}, K ) := \inf\big\{ \enormsm{\vct{x} - \vct{y}} : \vct{y} \in K \big\}$.  

A \term{convex cone} $C$ is a nonempty subset of $\R^d$ that satisfies
$$
\tau \cdot (\vct{x} + \vct{y}) \in C
\quad\text{for all $\tau > 0$ and $\vct{x}, \vct{y} \in C$.}
$$
We designate the family $\cC_d$ of all closed convex cones in $\R^d$.
A cone $C$ is \term{polyhedral} if it can be expressed as the intersection of a finite number of halfspaces:
$$
C = \bigcap_{i=1}^N \big\{ \vct{x} \in \R^d : \ip{ \vct{u}_i }{ \vct{x} } \geq 0 \big\}
\quad\text{for some $\vct{u}_i \in \R^d$.}
$$
For each cone $C \in \cC_d$, we define the \term{polar cone} $C^\polar \in \cC_d$ via the formula
$$
C^\polar := \big\{ \vct{u} \in \R^d : \ip{ \vct{u} }{ \vct{x} } \leq 0
\text{ for all $\vct{x} \in C$} \big\}.
$$
The polar of a polyhedral cone is always polyhedral.

We introduce the \term{metric projector} $\Proj_C$ onto a cone $C \in \cC_d$ by the formula
\begin{equation} \label{eqn:metric-proj}
\Proj_C : \R^d \to C
\qtq{where}
\Proj_C(\vct{x}) := \argmin\big\{ \enormsm{ \vct{x} - \vct{y} }^2 : \vct{y} \in C \big\}.
\end{equation}
The metric projector onto a closed convex cone is a nonnegatively homogeneous function:
$$
\Proj_C(\tau \vct{x}) = \tau \cdot \Proj_C(\vct{x})
\quad\text{for all $\tau \geq 0$ and $\vct{x} \in \R^d$.}
$$
The squared norm of the metric projection is a differentiable function:
\begin{equation} \label{eqn:grad-proj}
\grad \enorm{ \Proj_C(\vct{x}) }^2 = 2\, \Proj_C(\vct{x})
\quad\text{for all $\vct{x} \in \R^d$.}
\end{equation}
This result follows from~\cite[Thm.~2.26]{RW:98}.

We conclude with the basic notation concerning probability.  We write $\Prob$ for the probability of an event and $\Expect$ for the expectation operator.  The symbol $\sim$ denotes equality of distribution.  We reserve the letter $\vct{g}$ for a standard Gaussian vector, and $\vct{\theta}$ denotes a vector uniformly distributed on the sphere.  The dimensions are determined by context.

\section{The intrinsic volumes of a convex cone}
\label{sec:intvols}

We begin with an introduction to the conic intrinsic volumes
that is motivated by the treatment in~\cite{Ame:11}.
To each closed convex cone,
we can assign a sequence of intrinsic volumes. For polyhedral cones, these functionals
have a clear geometric meaning, so we start with
the definition for this special case.

\begin{definition}[Intrinsic volumes of a polyhedral cone] \label{def:intvols}
Let $C \in \cC_d$ be a polyhedral cone.  For $k = 0, 1, 2, \dots, d$, the \term{conic intrinsic volume} $v_k(C)$ is the quantity
$$
v_k(C) := \Prob\big\{ \text{$\Proj_C(\vct{g})$ lies in the relative interior of a $k$-dimensional face of $C$} \big\}.
$$
The metric projector $\Proj_C$ onto the cone is defined in~\eqref{eqn:metric-proj},
and the random vector $\vct{g}$ is drawn from the standard Gaussian distribution on $\R^d$.
\end{definition}

As explained in Section~\ref{sec:relat-betw-conic}, we can equip the set $\cC_d$
with the \term{conic Hausdorff metric} to form a compact metric space.  The
polyhedral cones form a dense subset of $\cC_d$, so it is natural
to use approximation to extend the definition of the intrinsic volumes
to nonpolyhedral cones.

\begin{definition}[Intrinsic volumes of a closed convex cone] \label{def:intvols-gen}
Let $C \in \cC_d$ be a closed convex cone.  Consider any sequence $( C_i )_{i \in \mathbb{N}}$ of polyhedral cones in $\cC_d$ where $C_i \to C$ in the
conic Hausdorff metric.  Define
\begin{equation} \label{eqn:intvol-limit}
v_k(C) := \lim_{i \to \infty} v_k(C_i)
\quad\text{for $k = 0, 1, 2, \dots, d$.}
\end{equation}
The geometric functionals $v_k : \cC_d \to [0, 1]$ are called
\term{conic intrinsic volumes}.
\end{definition}

\noindent
See Section~\ref{sec:continuity-intvols} for a proof that the limit in~\eqref{eqn:intvol-limit} is well defined.  The reader should be aware that the geometric interpretation of intrinsic volumes from Definition~\ref{def:intvols} breaks down for general cones because the limiting process does not preserve facial structure.

The conic intrinsic volumes have some remarkable properties.
Fix the ambient dimension $d$, and let $C \in \cC_d$ be a closed convex
cone in $\R^d$.  \textbf{The intrinsic volumes are...}

\begin{enumerate} \setlength{\itemsep}{2mm}
\item	{\textbf{Intrinsic.}}
The intrinsic volumes do not depend on the dimension of the space $\R^d$ in which the cone $C$ is embedded.   That is, for each natural number $r$,
$$
v_k( C \times \{\vct{0}_r\} ) =
	\begin{dcases}
	v_k(C), & 0 \leq k \leq d \\
	0, & d < k \leq d + r.
	\end{dcases}
$$

\item	{\textbf{Volumes.}}  Let $\gamma_d$ denote the standard Gaussian measure on $\R^d$.  Then $v_d(C) = \gamma_d(C)$ and $v_0(C) = \gamma_d(C^\polar)$ where $C^\polar$ denotes the polar cone.
The other intrinsic volumes, however, do not admit such a clear interpretation.

\item	{\textbf{Rotation invariant.}}  For each $d \times d$ orthogonal matrix $\mtx{Q}$, we have $v_k(\mtx{Q} C) = v_k(C)$.

\item {\textbf{Continuous.}} If  \(C_i \to C \) in the conic Hausdorff metric, then \(v_k(C_i) \to v_k(C)\). 

\item	{\textbf{A distribution.}}  The intrinsic volumes form a probability distribution on $\{ 0, 1, 2, \dots, d \}$.  That is,
$$
v_k(C) \geq 0
\qtq{and}
\sum_{j=0}^d v_j(C) = 1.
$$

\item	{\textbf{Indicators of dimension for a subspace.}}  For any $j$-dimensional subspace $L_j \subset \R^d$, we have
$$
v_k(L_j) = \begin{cases} 1, & k = j \\ 0, & k \neq j. \end{cases}
$$

\item	{\textbf{Reversed under polarity.}}  The intrinsic volumes of the polar cone $C^\polar$ satisfy
$$
v_k(C^\polar) = v_{d-k}(C).
$$
\end{enumerate}

\noindent
These claims follow from Definition~\ref{def:intvols-gen} using facts from
Sections~\ref{sec:convex-analysis} and~\ref{sec:backgr-conv-cones}
about the geometry of convex cones.

\begin{remark}[Notation for intrinsic volumes]
The notation $v_k$ for the $k$th intrinsic volume does not specify the ambient dimension.  This convention is justified because the intrinsic volumes of a cone do not depend on the embedding dimension.
\end{remark}

\section{A generalized Steiner formula for cones}
\label{sec:gen-steiner}

As we have seen, the intrinsic volumes of a cone form a probability distribution.  Probabilistic methods offer a powerful technique for studying the intrinsic volumes.  To pursue this idea, we want access to moments and other statistics of the sequence of intrinsic volumes.  We acquire this information using a general Steiner formula for cones.  
\subsection{The master Steiner formula}
\label{sec:master-steiner}

Let us introduce a class of Gaussian integrals.  Fix a Borel measurable bivariate function $f : \R_+^2 \to \R$.  Consider the geometric functional
\begin{equation} \label{eq:phi-f}
\phi_f : \cC_d \to \R
\qtq{where}
\phi_f(C) := \Expect \big[ f\big( \enormsm{ \Proj_C(\vct{g}) }^2, \ \enormsm{\Proj_{C^\polar}(\vct{g}) }^2 \big) \big].
\end{equation}
As usual,  $\vct{g}\in \R^d$  is a standard Gaussian vector, and the expectation is interpreted as a Lebesgue integral.
We can develop an elegant expansion of $\phi_f$
in terms of the conic intrinsic volumes.

\begin{theorem}[Master Steiner formula for cones] \label{thm:general-steiner}
Let $f : \R^2_+ \to \R$ be a Borel measurable function.
Then the geometric functional $\phi_f$ defined in~\eqref{eq:phi-f}
admits the expression
\begin{equation} \label{eq:gen-steiner}
\phi_f(C) = \sum_{k=0}^d \phi_f( L_k ) \cdot v_k(C)
\quad\text{for $C \in \cC_d$}
\end{equation}
provided that all the expectations in~\eqref{eq:gen-steiner} are finite.
Here, $L_k$ denotes a $k$-dimensional subspace of $\R^d$ and the conic intrinsic volumes $v_k$ are introduced in Definition~\ref{def:intvols-gen}.
\end{theorem}

The coefficients \(\phi_f(L_k)\) in the expression~\eqref{eq:gen-steiner} have an alternative form that is convenient for computations.  Let $L_k$ be an arbitrary $k$-dimensional subspace of $\R^d$.  According to the definition~\eqref{eq:phi-f} of the functional $\phi_f$,
$$
\phi_f(L_k) = \Expect\big[ f\big( \enormsm{\Proj_{L_k}(\vct{g})}^2, \
	\enormsm{\Proj_{{L_k}^\polar}(\vct{g})}^2 \big) \big].
$$
The marginal property of the standard Gaussian vector $\vct{g}$ ensures that $\Proj_{L_k}(\vct{g})$ and $\Proj_{{L_k}^\polar}(\vct{g})$ are independent standard Gaussian vectors supported on $L_k$ and ${L_k}^\polar$.  Thus,
$\enormsm{\Proj_{L_k}(\vct{g})}^2$ and $\enormsm{\Proj_{{L_k}^\polar}(\vct{g})}^2$ are independent chi-square random variables with $k$ and $d-k$ degrees of freedom respectively.  Note the convention that a chi-square variable with zero degrees of freedom is identically zero.  In view of this fact, we have the following equivalent of Theorem~\ref{thm:general-steiner}.

\begin{corollary} \label{cor:general-steiner}
Instate the hypotheses and notation of Theorem~\ref{thm:general-steiner}.
Let $\big\{ X_0, \dots, X_d \big\}$ be an independent sequence of random variables where $X_k$ has the chi-square distribution with $k$ degrees of freedom, and let $\big\{ X'_{\smash{0}}, \dots, X'_{\smash{d}} \big\}$ be an independent copy of this sequence.  Then
\begin{equation} \label{eq:simple-steiner}
\phi_f(C)
= \sum_{k=0}^d \Expect\big[ f \big(X_k^{\phantom.}, \ X'_{d-k} \big) \big] \cdot v_k(C)
\quad\text{for $C \in \cC_d$.}
\end{equation}
\end{corollary}

Corollary~\ref{cor:general-steiner} has an appealing probabilistic consequence.
For every cone $C \in \cC_d$, the random variable $\enormsm{\Proj_C(\vct{g})}^2$ is a mixture of
chi-square random variables $X_0, \dots, X_d$ where the mixture coefficients $v_k(C)$ are determined solely by the cone.
This fact corresponds with a classical observation from the field of constrained statistical inference, where the random variate $\enormsm{\Proj_C(\vct{g})}^2$ is known as the chi-bar-squared statistic~\cite[Sec.~3.4]{SilSen:05}.

We outline the proof of Theorem~\ref{thm:general-steiner} in
Sections~\ref{sec:backgr-conv-cones} and~\ref{sec:proof-gener-steiner}.
The argument involves techniques that are already familiar to experts.  Many of the core ideas appear in McMullen's influential paper~\cite{McM:75}.  A similar approach has been used in hyperbolic integral geometry~\cite[p.~242]{San:80}; see also the proof of~\cite[Thm.~6.5.1]{SchWei:08}.  The main technical novelty is our method for showing that the conic intrinsic volumes are continuous with respect to the conic Hausdorff metric.

\subsection{How big is the expansion of a cone?}
\label{sec:expansion}

The Euclidean Steiner formula~\eqref{eqn:euclid-steiner} describes the volume of a Euclidean expansion of a compact convex set.  Although it may not be obvious from the identity~\eqref{eq:gen-steiner}, the master Steiner formula contains information about the volume of an expansion of a convex cone.  This section explains the connection, which justifies our decision to call Theorem~\ref{thm:general-steiner} a Steiner formula.

First, we argue that there is a simple expression for the Gaussian measure of a Euclidean expansion of a convex cone.

\begin{proposition}[Gaussian Steiner formula] \label{prop2:gauss-stein}
For each cone $C \in \cC_d$ and each number $\lambda \geq 0$,
\begin{equation} \label{eq:gauss2-stein}
\gamma_d(C + \sqrt{\lambda} \, \mathsf{B}_d)
	= \Prob\big\{ \dist^2( \vct{g}, C ) \leq \lambda \big\}	
	= \sum_{k=0}^d \Prob\big\{X_{d-k} \leq \lambda \big\} \cdot v_k(C)
\end{equation}
where $\gamma_d$ is the standard Gaussian measure on $\R^d$ and
$\mathsf{B}_d$ is the unit ball in $\R^d$.  The random variable
$X_j$ follows the chi-square distribution with $j$ degrees of freedom.
\end{proposition}

\begin{proof}
The first identity in~\eqref{eq:gauss2-stein} is immediate.  For the second,
we appeal to Corollary~\ref{cor:general-steiner} with the function
$$
f(a,b) = \begin{cases} 1, & b \leq \lambda \\ 0, & \text{otherwise.} \end{cases}
$$
This step yields the relation
$$
\Prob\big\{ \dist^2(\vct{g}, C) \leq \lambda \big\}
	= \Prob\big\{ \enormsm{\Proj_{C^\polar}(\vct{g})}^2 \leq \lambda \big\}
	= \sum_{k=0}^d \Expect \big[ f \big( X_k, X'_{d-k} \big) \big]
	= \sum_{k=0}^d \Prob\big\{ X'_{d-k} \leq \lambda \big\} \cdot v_k(C).
$$
The first equality depends on the representation~\eqref{eqn:dist-proj} of the distance
to a cone in terms of the metric projector onto the polar cone.  The result~\eqref{eq:gauss2-stein} follows because $X'_{d-k}$ has the same distribution as $X_{d-k}$.
\end{proof}

We can also establish the spherical Steiner formula~\eqref{eqn:sphere-steiner} as a consequence of Theorem~\ref{thm:general-steiner} by replacing the Gaussian vector $\vct{g}$ in Proposition~\ref{prop2:gauss-stein} with a random vector $\vct{\theta}$ that is uniformly distributed on the Euclidean unit sphere.  This strategy leads to an expression for the proportion of the sphere subtended by an angular expansion of the cone.

\begin{proposition}[Spherical Steiner formula]
\label{prop2:sphere-stein}
For each cone $C \in \cC_d$ and each number $\lambda \in [0, 1]$, it holds that
\begin{equation} \label{eq:sph-stein-intv}
\Prob\big\{ \dist^2(\vct{\theta}, C) \leq \lambda \big\}
	= \sum_{k=0}^d \Prob\big\{ B_{k,d-k} \leq \lambda \big\} \cdot v_k(C).
\end{equation}
The random vector $\vct{\theta}$ is drawn from the uniform distribution on the unit sphere $\sphere{d-1}$ in $\R^d$, and the random variable $B_{j, \ell}$ follows the $\textsc{beta}\big(\half j, \ \half \ell \big)$ distribution.  
\end{proposition}

\begin{proof} When $\lambda = 1$, both sides of~\eqref{eq:sph-stein-intv} equal one.  For $\lambda < 1$,
we convert the spherical variable to a Gaussian using
the relation $\vct{\theta} \sim \vct{g} / \enormsm{\vct{g}}$.
It follows that
$$
\Prob\big\{ \dist^2(\vct{\theta}, C) \leq \lambda \big\}
	= \Prob\big\{ \enormsm{\Proj_{C^\polar}(\vct{g})}^2
	\leq \lambda (1 - \lambda)^{-1} \cdot \enormsm{\Proj_{C}(\vct{g})}^2 \big\}.
$$
This identity depends on the representation~\eqref{eqn:dist-proj} of the distance, the nonnegative homogeneity of the metric projector $\Proj_{C^\polar}$, and the Pythagorean relation~\eqref{eqn:pythag}.
Apply Corollary~\ref{cor:general-steiner} with the function
$$
f(a,b) = \begin{cases} 1, & b \leq \lambda(1-\lambda)^{-1} a \\
	0, & \text{otherwise}. \end{cases}
$$
To finish, we recall the geometric interpretation~\cite{Art:02} of the beta random variable: $B_{j,d-j} \sim \enormsm{\Proj_{L_j}(\vct{\theta})}^2$ where $L_j$ is a $j$-dimensional subspace of $\R^d$.
\end{proof}

Neither the Gaussian Steiner formula~\eqref{eq:gauss2-stein} nor the spherical Steiner formula~\eqref{eq:sph-stein-intv} is new.  The spherical formula is classical~\cite{All:48}, while the Gaussian formula has antecedents in the statistics literature~\cite{SilSen:05,Tay:06}.  There is novelty, however, in our method of condensing both results from the master Steiner formula~\eqref{eq:gen-steiner}.

\section{Probabilistic analysis of intrinsic volumes}
\label{sec:prob-intvols}

In this section, we apply probabilistic methods
to the conic intrinsic volumes.  Our main tool is the master Steiner
formula, Theorem~\ref{thm:general-steiner}, which we use repeatedly to
convert statements about the intrinsic volumes of a cone into statements
about the projection of a Gaussian vector onto the cone.  We may then apply
methods from Gaussian analysis to study this random variable.

\subsection{The intrinsic volume random variable}
\label{sec:intvol-rv}

Definitions~\ref{def:intvols} and~\ref{def:intvols-gen} make it clear that the intrinsic volumes form a probability distribution.  This observation suggests that it would be fruitful to analyze the intrinsic volumes using techniques from probability.  We begin with the key definition.

\begin{definition}[Intrinsic volume random variable]
Let $C \in \cC_d$ be a closed convex cone.  The \term{intrinsic volume random variable} $V_C$ has the distribution
\begin{equation*} \label{eqn:intvol-rv}
\Prob\big\{ V_C = k \big\}
	= v_k(C)
	\quad\text{for $k = 0,1,2, \dots, d$.}
\end{equation*}
\end{definition}

Notice that the intrinsic volume random variable of a cone $C \in \cC_d$ and its polar have a tight relationship:
$$
\Prob\big\{ V_{C^\polar} = k \big\}
	= v_k(C^\polar)
	= v_{d-k}(C)
	= \Prob\big\{ V_C = d - k \big\}
$$
because polarity reverses the sequence of intrinsic volumes.  In other words, $V_{C^\polar} \sim d - V_C$.

\subsection{The statistical dimension of a cone}
\label{sec:intvol-mean}

The expected value of the intrinsic volume random variable $V_C$ has a distinguished place in the theory because $V_C$ concentrates sharply about this point.  In anticipation of this result, we glorify the expectation of $V_C$ with its own name and notation.

\begin{definition}[Statistical dimension~\protect{\cite[Sec.~5.3]{ALMT:13}}] \label{def:sdim}
The \term{statistical dimension} $\delta(C)$ of a cone $C \in \cC_d$ is the quantity
\begin{equation} \label{eqn:sdim}
\sdim(C) := \Expect[ V_C ] = \sum_{k=0}^d k \, v_k(C).
\end{equation}
\end{definition}

The statistical dimension of a cone really is a measure of its dimension.  In particular,
\begin{equation} \label{eqn:sdim-dim}
\sdim( L ) = \dim( L )
\quad\text{for each subspace $L \subset \R^d$.}
\end{equation}
In fact, the statistical dimension is the canonical extension
of the dimension of a subspace to the class of convex cones~\cite[Sec.~5.3]{ALMT:13}.
By this, we mean that the statistical dimension
is the only rotation invariant, continuous, localizable valuation on $\cC_d$ that satisfies~\eqref{eqn:sdim-dim}.
See~\cite[p.~254 and Thm.~6.5.4]{SchWei:08} for further
information about the unexplained technical terms.

The statistical dimension interacts beautifully with the polarity operation.  In particular,
\begin{equation*} \label{eqn:totality}
\delta(C) + \delta(C^\polar)
	= \Expect[ V_C ] + \Expect[ V_{C^\polar} ]
	= \Expect[ V_C ] + \Expect[ d - V_C ]
	= d.
\end{equation*}
This formula allows us to evaluate the statistical dimension for an important class of cones.  A closed convex cone $C$ is \term{self-dual} when it satisfies the identity $C = - C^\polar$.  Examples include the nonnegative orthant, the second-order cone, and the cone of positive-semidefinite matrices.  We have the identity
\begin{equation} \label{eqn:self-dual}
\sdim(C) = \half d
\quad \text{for a self-dual cone $C \in \cC_d$.}
\end{equation}

The statistical dimension of a cone can be expressed in terms
of the projection of a standard Gaussian vector onto the
cone~\cite[Prop.~5.11]{ALMT:13}.  The master Steiner
formula gives an easy proof of this result.

\begin{proposition}[Statistical dimension] \label{prop:sdim}
Let $C \in \cC_d$ be a closed convex cone.  Then
$$
\delta(C) = \Expect \big[ \enormsm{\Proj_C(\vct{g})}^2 \big].
$$
\end{proposition}

The identity in Proposition~\ref{prop:sdim} can be used to evaluate
the statistical dimension for many cones of interest.
See~\cite[Sec.~4]{ALMT:13} for details and examples.

\begin{proof}
The master Steiner formula, Corollary~\ref{cor:general-steiner},
with function $f(a,b) = a$ states that
$$
\Expect\big[ \enormsm{\Proj_C(\vct{g})}^2 \big]
	= \sum_{k=0}^d \Expect[ X_k ] \cdot v_k(C)
	= \sum_{k=0}^d k \, v_k(C)
	= \delta(C).
$$
Indeed, a chi-square variable $X_k$ with $k$ degrees of freedom has expectation $k$.
\end{proof}

\subsection{The variance of the intrinsic volumes}
\label{sec:intvol-var}

The variance of the intrinsic volume random variable tells us how tightly
the intrinsic volumes cluster around their mean value.  We can find an
explicit expression for the variance in terms of the projection of a
Gaussian vector onto the cone.

\begin{proposition}[Variance of the intrinsic volumes] \label{prop:variance}
Let $C \in \cC_d$ be a closed convex cone.  Then
\begin{align} 
\Var[ V_C ] &= \Var\big[ \enormsm{\Proj_C(\vct{g})}^2 \big]
	- 2 \delta(C) \label{eqn:variance-1}\\
	&= \Var\big[ \enormsm{\Proj_{C^\polar}(\vct{g})}^2 \big]
	- 2 \delta(C^\polar) = \Var[V_{C^\polar}]. \label{eqn:variance-2}
\end{align}
\end{proposition}

Proposition~\ref{prop:variance} leads to exact formulas for the variance of the intrinsic volume sequence in several interesting cases; Section~\ref{sec:examples} contains some worked examples.

\begin{proof}
By definition, the variance satisfies
$$
\Var[ V_C ] = \Expect\big[ V_C^2 \big] - \big(\Expect[ V_C ] \big)^2
	= \Expect\big[ V_C^2 \big] - \delta(C)^2.
$$
To obtain the expectation of $V_C^2$,
we invoke the master Steiner formula, Corollary~\ref{cor:general-steiner},
with the function $f(a,b) = a^2$ to obtain
$$
\Expect\big[ \enormsm{\Proj_C(\vct{g})}^4 \big]
	= \sum_{k=0}^d \Expect\big[ X_k^2 \big] \cdot v_k(C)
	= \sum_{k=0}^d k^2 v_k(C) + 2 \sum_{k=0}^d k \, v_k(C)
	= \Expect\big[ V_C^2 \big] + 2\delta(C). 
$$
Indeed, the raw second moment of a chi-square random variable $X_k$ with $k$ degrees of freedom equals $k^2 + 2k$.
Combine these two displays to reach
$$
\Var[ V_C ] = \Expect\big[ \enormsm{\Proj_C(\vct{g})}^4 \big] - \delta(C)^2 - 2 \delta(C)
	= \Expect\big[ \enormsm{\Proj_C(\vct{g})}^4 \big] - \big( \Expect\big[ \enormsm{\Proj_C(\vct{g})}^2 \big] \big)^2 - 2 \delta(C)
$$
where the second identity follows from Proposition~\ref{prop:sdim}.  Identify the variance of $\enormsm{\Proj_C(\vct{g})}^2$ to complete the proof of~\eqref{eqn:variance-1}.
To establish~\eqref{eqn:variance-2}, note that
$\Var[ V_C ] = \Var[ d - V_C ] = \Var[ V_{C^\polar} ]$,
and then apply~\eqref{eqn:variance-1}
to the random variable $V_{C^\polar}$.
\end{proof}

\subsection{A bound for the variance of the intrinsic volumes}
\label{sec:intvol-var-bd}

Proposition~\ref{prop:variance} also allows us to produce a general bound on the variance of the intrinsic volumes of a cone.

\begin{theorem}[Variance bound for intrinsic volumes] \label{thm:variance-bd}
Let $C \in \cC_d$ be a closed convex cone. Then
\begin{equation*} \label{eqn:variance-bd}
\Var[ V_C ] \leq 2 \, \big( \delta(C) \wedge \delta(C^\polar) \big).
\end{equation*}
The operator $\wedge$ returns the minimum of two numbers.
\end{theorem}

The example in Section~\ref{sec:circular} demonstrates that the constant two in~\eqref{eqn:variance-bd} cannot be reduced in general.

\begin{proof}
To bound the variance of $V_C$, we plan to invoke the Gaussian Poincar{\'e} inequality~\cite[Thm.~1.6.4]{Bog:98} to control the variance of $\enormsm{\Proj_C(\vct{g})}^2$.  This inequality states that
$$
\Var[ H(\vct{g}) ] \leq \Expect\big[ \enormsm{ \grad H(\vct{g}) }^2 \big]
$$
for any function $H : \R^d \to \R$ whose gradient is square-integrable with respect to the standard Gaussian measure.
We apply this result to the function
$$
H(\vct{x}) = \enormsm{\Proj_C(\vct{x})}^2
\qtq{with}
\enormsm{\grad H(\vct{x})}^2 = 4 \enormsm{ \Proj_C(\vct{x}) }^2.
$$
The gradient calculation is justified by~\eqref{eqn:grad-proj}.
We determine that
$$
\Var\big[ \enormsm{\Proj_C(\vct{g})}^2 \big]
	\leq 4 \Expect \big[ \enormsm{\Proj_C(\vct{g})}^2 \big]
	= 4 \delta(C)
$$
where the second identity follows from Proposition~\ref{prop:sdim}.  Introduce this inequality into~\eqref{eqn:variance-1} to see that $\Var[V_C] \leq 2\delta(C)$.
We can apply the same argument to see that
$$
\Var\big[ \enormsm{\Proj_{C^\polar}(\vct{g})}^2 \big] \leq 4 \delta(C^\polar).
$$
Substitute this bound into~\eqref{eqn:variance-2} to conclude that $\Var[V_C] \leq 2 \delta(C^\polar)$.
\end{proof}

In principle, a random variable taking values in $\{0,1,2, \dots, d\}$ can have variance larger than $d^2/3$---consider the uniform random variable.  In contrast, Theorem~\ref{thm:variance-bd} tells us that the variance of the intrinsic volume random variable $V_C$ cannot exceed $d$ for any cone $C$.  This observation has consequences for the tail behavior of $V_C$.  Indeed, Chebyshev's inequality implies that
\begin{equation*} \label{eqn:chebyshev}
\Prob\bigg\{ \abs{ V_C - \delta(C) } > \lambda \sqrt{\delta(C)} \bigg\}
	\leq \frac{\Var[V_C]}{\lambda^2 \delta(C)}
	\leq \frac{2}{\lambda^2}.
\end{equation*}
That is, most of the mass of $V_C$ is located near the statistical dimension.

\subsection{Exponential moments of the intrinsic volumes}
\label{sec:intvol-exp}

In the previous section, we discovered that the intrinsic volume random variable $V_C$ is often close to its mean value.  This observation suggests that $V_C$ might exhibit stronger concentration.  A standard method for proving concentration inequalities for a random variable is to calculate its exponential moments.  The master Steiner formula allows us to accomplish this task.

\begin{proposition}[Exponential moments of the intrinsic volumes] \label{prop:exp-mom}
Let $C \in \cC_d$ be a closed convex cone.  For each parameter $\zeta \in \R$,
$$
\Expect{} \econst^{\zeta V_C}
	= \Expect{} \econst^{\xi \, \enormsm{\Proj_C(\vct{g})}^2}
	\qtq{where}
	\xi = \half \big(1 - \econst^{-2\zeta}\big).
$$
\end{proposition}

\begin{proof}
Fix a number $\xi < \half$.  With the choice $f(a,b) = \econst^{\xi a}$, Corollary~\ref{cor:general-steiner} shows that
$$
\Expect{} \econst^{\xi \, \enormsm{\Proj_C(\vct{g})}^2}
	= \sum_{k=0}^d \Expect{}\big[ \econst^{\xi X_k} \big] \cdot v_k(C)
	= \sum_{k=0}^d (1 - 2\xi)^{-k/2} v_k(C)
	= \sum_{k=0}^d \econst^{\zeta k} v_k(C)
	= \Expect{} \econst^{\zeta V_C}.
$$
We have used the familiar formula for the exponential moments of a chi-square random variable $X_k$ with $k$ degrees of freedom.  The penultimate identity follows from the change of variables $\zeta = - \half \log(1 - 2\xi)$, which establishes a bijection between $\xi < \half$ and $\zeta \in \R$.
\end{proof}

\begin{remark}[Conic Wills functional]
Proposition~\ref{prop:exp-mom} leads to a geometric description of the generating function of the intrinsic volumes:
\begin{equation} \label{eqn:wills}
W_C(\lambda)
	:= \lambda^d \Expect \exp\left( \frac{1-\lambda^2}{2} \cdot\dist^2(\vct{g}, C) \right)
	= \sum_{k=0}^d \lambda^k v_k(C)
	\quad\text{for $\lambda > 0$.}
\end{equation}
To see why this is true, use the representation~\eqref{eqn:dist-proj} of the distance, and apply Proposition~\ref{prop:exp-mom} with
$\zeta = -\log \lambda$ to confirm that
$$
W_C(\lambda) =
	\lambda^d \Expect \exp\left( \frac{1 - \lambda^2}{2} \cdot \enormsm{\Proj_{C^\polar}(\vct{g})}^2 \right)
	= \lambda^d \sum_{k=0}^d \lambda^{-k} v_k(C^\polar)
	= \sum_{k=0}^d \lambda^{k} v_k(C).
$$
We have applied the fact that polarity reverses intrinsic volumes to reindex the sum.  The function $W_C$ can be viewed as a conic analog of the Wills functional~\cite{Wil:73} from Euclidean geometry.
\end{remark}

\subsection{A bound for the exponential moments of the intrinsic volumes}
\label{sec:intvol-exp-bd}

Proposition~\ref{prop:exp-mom} allows us to obtain an excellent bound for the exponential moments of $V_C$.  In the next section, we use this result to develop concentration inequalities for the intrinsic volumes.

\begin{theorem}[Exponential moment bound for intrinsic volumes] \label{thm:exp-bd}
Let $C \in \cC_d$ be a closed convex cone.  For each parameter $\zeta \in \R$,
\begin{align}
\Expect \econst^{\zeta (V_C - \delta(C))}
	&\leq \exp\left( \frac{\econst^{2\zeta} - 2 \zeta - 1}{2} \cdot \delta(C) \right)
	\label{eqn:exp-bd-1}, \quad\text{and} \\
\Expect \econst^{\zeta (V_C - \delta(C))}
	&\leq \exp\left( \frac{\econst^{-2\zeta} + 2\zeta - 1}{2} \cdot \delta(C^\polar) \right).
	\label{eqn:exp-bd-2}
\end{align}
\end{theorem}

The major technical challenge is to bound the exponential moments of the random variable $\enormsm{\Proj_C(\vct{g})}^2$.  The following lemma provides a sharp
estimate for the exponential moments.  It improves on an earlier
result~\cite[Sublem.~D.3]{ALMT:13}.

\begin{lemma} \label{lem:exp-bd}
Let $C \in \cC_d$ be a closed convex cone.  For each parameter $\xi < \half$,
$$
\Expect \econst^{\xi \, ( \enormsm{\Proj_C(\vct{g})}^2 - \delta(C) )}
	\leq \exp\left( \frac{2\xi^2 \delta(C)}{1 - 2\xi} \right).
$$
\end{lemma}

\begin{proof}
Define the zero-mean random variable
$$
Z := \enormsm{\Proj_C(\vct{g})}^2 - \delta(C).
$$
Introduce the moment generating function $m(\xi) := \Expect{} \econst^{\xi Z}$.  Our aim is to bound $m(\xi)$.  Before we begin, it is helpful to note a few properties of the moment generating function.  First, the derivative $m'(\xi) = \Expect\big[ Z\econst^{\xi Z} \big]$ whenever $\xi < \half$.  By direct calculation, $\log m(0) = 0$.  Furthermore, l'H{\^ o}pital's rule shows that $\lim_{\xi \to 0} \xi^{-1} \log m(\xi) = 0$ because the random variable $Z$ has zero mean.

The argument is based on the Gaussian logarithmic Sobolev inequality~\cite[Thm.~1.6.1]{Bog:98}.  One version of this result states that
\begin{equation} \label{eqn:log-sob}
\Expect \left[ H(\vct{g} ) \cdot \econst^{H(\vct{g})} \right]
	- \Expect\left[ \econst^{H(\vct{g})} \right] \log{} \Expect\left[  \econst^{H(\vct{g})} \right]
	\leq \frac{1}{2} \Expect \left[ \enormsm{\grad H(\vct{g})}^2 \cdot \econst^{H(\vct{g})} \right]
\end{equation}
for any differentiable function $H : \R^d \to \R$ such that the expectations in~\eqref{eqn:log-sob} are finite.  We apply this result to the function
$$
H(\vct{x}) = \xi \, \left[ \enormsm{\Proj_C(\vct{x})}^2 - \delta(C) \right]
\qtq{with}
\enormsm{\grad H(\vct{x})}^2 = 4 \xi^2 \enormsm{\Proj_C(\vct{x})}^2.
$$
The gradient calculation is justified by~\eqref{eqn:grad-proj}.  Notice that
$$
H(\vct{g}) = \xi Z
\qtq{and}
\enormsm{\grad H(\vct{g})}^2 = 4\xi^2(Z + \delta(C)).
$$
Therefore, the logarithmic Sobolev inequality~\eqref{eqn:log-sob} delivers the relation
$$
\xi \cdot \Expect \big[ Z \econst^{\xi Z} \big] - \Expect\big[ \econst^{\xi Z} \big]
	\log{} \Expect\big[ \econst^{\xi Z} \big]
	\leq 2 \xi^2 \cdot \Expect\big[ Z\econst^{\xi Z} \big] + 2\xi^2 \delta(C) \cdot \Expect \big[ \econst^{\xi Z} \big]
	\quad\text{for $\xi < \half$.}
$$
We can rewrite the last display as a differential inequality for the moment generating function:
\begin{equation} \label{eqn:diff-ineq-0}
\xi m'(\xi) - m(\xi) \log m(\xi) \leq 2\xi^2 m'(\xi) + 2\delta(C) \cdot \xi^2 m(\xi)
\quad\text{for $\xi < \half$.}
\end{equation}
The requirement on $\xi$ is necessary and sufficient to ensure that $m(\xi)$ and $m'(\xi)$ are finite.  To complete the proof, we just need to solve this differential inequality.

We follow the argument from~\cite[Thm.~5]{BLM03:Concentration-Inequalities}.  Divide the inequality~\eqref{eqn:diff-ineq-0} by the positive number $\xi^2 m(\xi)$ to reach
$$
\frac{1}{\xi} \cdot \frac{m'(\xi)}{m(\xi)} - \frac{1}{\xi^2} \log m(\xi)
	\leq 2 \cdot \frac{m'(\xi)}{m(\xi)} + 2\delta(C)
	\quad\text{for $\xi \in (-\infty, 0) \cup \big(0, \half \big)$.}
$$
The left- and right-hand sides of this relation are exactly integrable:
\begin{equation} \label{eqn:diff-ineq}
\frac{\diff{}}{\diff{s}} \left[ \frac{1}{s} \log m(s) \right]
	\leq 2 \cdot \frac{\diff{}}{\diff{s}} \left[ \log m(s) + 2\delta(C) \cdot s\right]
	\quad\text{for $s \in (-\infty, 0) \cup \big(0, \half \big)$.}
\end{equation}
To continue, we first consider the case $0 < \xi < \half$.  Integrate the inequality~\eqref{eqn:diff-ineq} over the interval $s \in [0, \xi]$ using the boundary conditions $\log m(0) = 0$ and $\lim_{\xi \to 0} \xi^{-1} \log m(\xi) = 0$.  This step yields
$$
\frac{1}{\xi} \log m(\xi) \leq 2 \log m(\xi) + 2\delta(C) \cdot \xi
	\quad\text{for $0 < \xi < \half$.}
$$
Solve this relation for the moment generating function $m$ to obtain the bound
\begin{equation} \label{eqn:mgf-bd}
m(\xi) \leq \exp\left( \frac{2\xi^2 \delta(C)}{1 - 2\xi} \right)
\quad\text{for $0 \leq \xi < \half$.}
\end{equation}
The boundary case $\xi = 0$ follows from a direct calculation.  Next, we address the situation where $\xi < 0$.  Integrating~\eqref{eqn:diff-ineq} over the interval $[\xi, 0]$, we find that
$$
- \frac{1}{\xi} \log m(\xi)  \leq - 2 \log m(\xi) - 2 \delta(C) \cdot \xi
	\quad\text{for $\xi < 0$.}
$$
Solve this inequality for $m(\xi)$ to see that the bound~\eqref{eqn:mgf-bd} also holds in the range $\xi < 0$.  This observation completes the proof. 
\end{proof}

With Lemma~\ref{lem:exp-bd} at hand, we quickly reach Theorem~\ref{thm:exp-bd}.

\begin{proof}[Proof of Theorem~\ref{thm:exp-bd}]
We begin with the statement from Proposition~\ref{prop:exp-mom}.  Adding and subtracting multiples of $\delta(C)$ in the exponent, we obtain the relation
$$
\Expect{} \econst^{\zeta (V_C - \delta(C))} =
	\econst^{(\xi - \zeta) \, \delta(C)} \cdot
	\Expect{} \econst^{\xi \, (\enormsm{\Proj_C(\vct{g})}^2 - \delta(C))}.
$$
where $\xi = \half \big(1 - \econst^{-2\zeta} \big) < \half$.  Lemma~\ref{lem:exp-bd} controls the moment generating function on the right-hand side:
$$
\Expect{} \econst^{\zeta (V_C - \delta(C))} \leq
	\econst^{(\xi - \zeta) \, \delta(C)} \cdot
	\exp\left( \frac{2\xi^2 \delta(C)}{1 - 2\xi} \right).
$$
By a marvelous coincidence, the terms in the exponent collapse into a compact form:
$$
\xi - \zeta + \frac{2\xi^2}{1 - 2\xi} = \frac{\econst^{2 \zeta} - 2 \zeta -1 }{2}.
$$
Combine the last two displays to finish the proof of~\eqref{eqn:exp-bd-1}.  To obtain the second formula~\eqref{eqn:exp-bd-2}, note that
$$
\Expect{} \econst^{\zeta (V_C - \delta(C))}
	= \Expect{} \econst^{(-\zeta) (V_{C^\polar} - \delta(C^\polar))}
$$
because $V_{C} \sim d - V_{C^\polar}$ and $\delta(C) = d - \delta(C^\polar)$.  Now apply~\eqref{eqn:exp-bd-1} to the right-hand side. 
\end{proof}

\subsection{Concentration of intrinsic volumes}
\label{sec:intvol-conc}

The exponential moment bound from Theorem~\ref{thm:exp-bd} allows us to obtain concentration results for the sequence of intrinsic volumes of a convex cone.  The following corollary provides Bennett-type inequalities for the intrinsic volume random variable.  
\begin{corollary}[Concentration of the intrinsic volume random variable] \label{cor:VC-conc}
Let $C \in \cC_d$ be a closed convex cone.  For each $\lambda \geq 0$,
the intrinsic volume random variable $V_C$ satisfies the upper tail bound
\begin{align}
\Prob\big\{ V_C - \delta(C) \geq \lambda \big\}
	&\leq \exp\left( -  \frac{1}{2} \max\left\{ \delta(C) \cdot \psi\left(\frac{\lambda}{\delta(C)} \right),\ 
	\delta(C^\polar) \cdot \psi\left(\frac{-\lambda}{\delta(C^\polar)}\right) \right\} \right) \label{eq:upper-tail} \\
\intertext{and the lower tail bound}
\Prob\big\{ V_C - \delta(C) \leq - \lambda \big\}
	&\leq \exp\left( - \frac{1}{2} \max\left\{ \delta(C) \cdot \psi\left(\frac{-\lambda}{\delta(C)} \right),\ 
	\delta(C^\polar) \cdot \psi\left(\frac{\lambda}{\delta(C^\polar)}\right) \right\} \right). \label{eq:lower-tail}
\end{align}
The function $\psi(u) := (1+u)\log(1+u) - u$ for $u \geq - 1$ while $\psi(u) = \infty$ for $u < - 1$.
\end{corollary}

\begin{proof} The argument, based on the Laplace transform method, is standard.  For any $\zeta > 0$,
$$
\Prob\big\{ V_C - \delta(C) \geq \lambda \big\}
	\leq \econst^{-\zeta \lambda} \cdot \Expect{} \econst^{\zeta \, (V_C - \delta(C))}
	\leq \econst^{-\zeta \lambda} \cdot \exp\left( \frac{\econst^{2\zeta} - 2\zeta - 1}{2} \cdot \delta(C) \right)
$$
where we have applied the exponential moment bound~\eqref{eqn:exp-bd-1} from Theorem~\ref{thm:exp-bd}.  Minimize the right-hand side over $\zeta > 0$ to obtain the first branch of the maximum in~\eqref{eq:upper-tail}.  The second exponential moment bound~\eqref{eqn:exp-bd-2} leads to the second branch of the maximum in~\eqref{eq:upper-tail}.  The lower tail bound~\eqref{eq:lower-tail} follows from the same considerations.  For more details about this type of proof, see~\cite[Sec.~2.7]{BLM13:Concentration-Inequalities}.
\end{proof}

To understand the content of Corollary~\ref{cor:VC-conc}, it helps to make some further estimates.  Comparing Taylor series, we find that $\psi(u) \geq u^2/(2 + 2u/3)$.  This observation leads to a weaker form of the tail bounds~\eqref{eq:upper-tail} and~\eqref{eq:lower-tail}.  For $\lambda \geq 0$,
\begin{align*}
\Prob\big\{ V_C - \delta(C) \geq \lambda \big\}
	&\leq \exp\left( \frac{-\lambda^2/4}{(\delta(C) + \lambda / 3) \wedge
	(\delta(C^\polar) - \lambda / 3)} \right) \\
\Prob\big\{ V_C - \delta(C) \leq -\lambda \big\}
	&\leq \exp\left( \frac{-\lambda^2/4}{(\delta(C) - \lambda / 3) \wedge
	(\delta(C^\polar) + \lambda / 3)} \right).
\end{align*}
This pair of inequalities reflects the fact that the left tail of $V_C$ exhibits faster decay than the right tail when the statistical dimension $\delta(C)$ is small; the tail behavior is reversed when $\delta(C)$ is close to the ambient dimension.  For practical purposes, it seems better to combine these estimates into a single bound:
\begin{equation} \label{eqn:tail-combined}
\Prob\big\{ \abs{V_C - \delta(C)} \geq \lambda \big\}
	\leq 2 \exp \left( \frac{-\lambda^2/4}{(\delta(C) \wedge \delta(C^\polar))
	+ \lambda / 3} \right) \quad\text{for $\lambda \geq 0$.}
\end{equation}
This tail bound indicates that $V_C$ looks somewhat like a Gaussian variable with
mean $\delta(C)$ and variance $2 \, (\delta(C) \wedge \delta(C^\polar))$ or less.
This claim is consistent with Theorem~\ref{thm:variance-bd}.
The result~\eqref{eqn:tail-combined} improves
over~\cite[Thm.~6.1]{ALMT:13}, and we will
see an example in Section~\ref{sec:circular} that saturates the bound.

Our analysis suggests that the intrinsic volume sequence of a convex cone $C$
cannot exhibit very complicated behavior.  Indeed, the statistical dimension
$\delta(C)$ already tells us almost everything there is to know.  The only
large intrinsic volumes $v_k(C)$ are those where the index
$k$ is in the range $\delta(C) \pm {\rm const} \cdot
\sqrt{\delta(C) \wedge \delta(C^\polar)}$.
The consequence of this result for conic integral geometry
is that a cone with statistical dimension $\delta(C)$
behaves essentially like a subspace with approximate dimension $\delta(C)$.
See~\cite{ALMT:13} for more support of this point and its
consequences for convex optimization.

\section{Intrinsic volumes of product cones}
\label{sec:intvol-product}

Suppose that $C_1 \in \cC_{d_1}$ and $C_2 \in \cC_{d_2}$ are closed convex
cones.  We can form another closed convex cone by taking their direct product:
$$
C_1 \times C_2 := \big\{ (\vct{x}_1, \vct{x}_2) \in \R^{d_1+d_2} :
	\vct{x}_1 \in C_1 \text{ and } \vct{x}_2 \in C_2 \big\}
	\in \cC_{d_1+d_2}.
$$
The probabilistic methods of the last section are well suited to the analysis
of a product cone.  In this section, we compute the intrinsic
volumes of a product cone using these techniques.  Then we identify the mean,
variance, and concentration behavior of the intrinsic volume random variable
of a product cone.

\subsection{The product rule for intrinsic volumes}
\label{sec:product}

The intrinsic volumes of the product cone can be derived from the intrinsic
volumes of the two factors.

\begin{corollary}[Product rule for intrinsic volumes] \label{cor:product-rule}
Let $C_1 \in \cC_{d_1}$ and $C_2 \in \cC_{d_2}$ be closed convex cones.
The intrinsic volumes of the product cone $C_1 \times C_2$ satisfy
\begin{equation} \label{eqn:product}
v_k(C_1 \times C_2) = \sum_{i+j = k} v_i(C_1) \cdot v_j(C_2)
\quad\text{for $k = 0, 1, 2, \dots, d_1 + d_2$.}
\end{equation}
\end{corollary}

We present a short proof of Corollary~\ref{cor:product-rule} based on
the conic Wills functional~\eqref{eqn:wills}.
This approach echoes Hadwiger's method~\cite{Had:75}
for computing the Euclidean intrinsic volumes of a product of convex sets.

\begin{proof}
Let $\vct{g}_1 \in \R^{d_1}$ and $\vct{g}_2 \in \R^{d_2}$ be independent standard Gaussian vectors.  The direct product $(\vct{g}_1, \vct{g}_2)$ is a standard Gaussian vector on $\R^{d_1+d_2}$.  For each $\lambda > 0$, the definition~\eqref{eqn:wills} of the Wills functional gives
\begin{align*}
W_{C_1 \times C_2}(\lambda)
	&= \lambda^{d_1+d_2} \Expect{} \exp\left( \frac{1 - \lambda^2}{2}
	\cdot \dist^2\big( (\vct{g}_1,\vct{g}_2), C_1 \times C_2 \big) \right) \\
	&= \lambda^{d_1} \Expect{} \exp\left( \frac{1 - \lambda^2}{2} \cdot
	\dist^2( \vct{g}_1, C_1) \right)
	\cdot \lambda^{d_2} \Expect{} \exp\left( \frac{1 - \lambda^2}{2} \cdot \dist^2(
	\vct{g}_2, C_2) \right)
	= W_{C_1}(\lambda) \cdot W_{C_2}(\lambda).
\end{align*}
The second identity follows from the fact that the squared distance to a product cone equals the sum of the squared distances to the factors; we have also invoked the independence of the two standard Gaussian vectors to split the expectation.  Applying the relation~\eqref{eqn:wills} twice, we find that
$$
W_{C_1 \times C_2}(\lambda) = W_{C_1}(\lambda) \cdot W_{C_2}(\lambda)
	= \left( \sum_{i=0}^{d_1} \lambda^i v_i(C_1) \right)
	\left( \sum_{j=0}^{d_2} \lambda^j v_j(C_2) \right)
	= \sum_{k=0}^{d_1+d_2} \lambda^k \sum_{i + j = k} v_i(C_1) \cdot v_j(C_2).
$$
But~\eqref{eqn:wills} also shows that
$$
W_{C_1 \times C_2}(\lambda) = \sum_{k=0}^{d_1+d_2} \lambda^k v_k(C_1 \times C_2).
$$
Comparing coefficients in these two polynomials, we arrive at the relation~\eqref{eqn:product}.
\end{proof}

\subsection{Concentration of the intrinsic volumes of a product cone}

We can employ the probabilistic techniques from
Section~\ref{sec:prob-intvols} to collect information about
the intrinsic volumes of a product cone.
Let $C_1 \in \cC_{d_1}$ and $C_2 \in \cC_{d_2}$ be two cones,
and consider \emph{independent} random variables $V_{C_1}$ and
$V_{C_2}$ whose distributions are given by the intrinsic volumes
of $C_1$ and $C_2$.  In view of Corollary~\ref{cor:product-rule},
$$
v_k(C_1 \times C_2) = \sum_{i+j=k} v_i(C_1) \cdot v_j(C_2)
	= \Prob\big\{ V_{C_1} + V_{C_2} = k \big\}
	\quad\text{for $k = 0, 1, 2, \dots, d_1 + d_2$.}
$$
In other words, the intrinsic volume random variable
$V_{C_1 \times C_2}$ of the product cone has the distribution
\begin{equation} \label{eqn:VC-prod}
V_{C_1 \times C_2} \sim V_{C_1} + V_{C_2}.
\end{equation}
This observation allows us to compute the statistical dimension
of the product cone:
\begin{equation} \label{eqn:sdim-prod}
\delta( C_1 \times C_2 ) = \Expect \big[ V_{C_1 \times C_2} \big]
	= \delta(C_1) + \delta(C_2).
\end{equation}
Of course, we can also derive~\eqref{eqn:sdim-prod} directly from
Proposition~\ref{prop:sdim}.
A more interesting consequence is the following expression for the
variance of the intrinsic volumes:
\begin{equation} \label{eqn:var-prod}
\Var\big[ V_{C_1 \times C_2} \big] = \Var[V_{C_1}] + \Var[V_{C_2}]
	\leq 2 \, \left[ \big(\delta(C_1) \wedge \delta(C_1^\polar) \big)
	+ \big(\delta(C_2) \wedge \delta(C_2^\polar)\big) \right].
\end{equation}
The inequality follows from Theorem~\ref{thm:variance-bd}.
With some additional effort, we can develop a concentration result
for the intrinsic volumes of a product cone that matches the
variance bound~\eqref{eqn:var-prod}.

\begin{corollary}[Concentration of intrinsic volumes for a product cone]
Let $C_1 \in \cC_{d_1}$ and $C_2 \in \cC_{d_2}$ be closed convex cones.
For each $\lambda \geq 0$,
$$
\Prob\big\{ \abs{ V_{C_1 \times C_2} - \delta(C_1 \times C_2) } \geq \lambda \big\}
	\leq 2 \, \exp\left( \frac{-\lambda^2/4}{\sigma^2 + \lambda/3} \right)
	\qtq{where}
	\sigma^2 := \big(\delta(C_1) \wedge \delta(C_1^\polar) \big)
	+ \big(\delta(C_2) \wedge \delta(C_2^\polar)\big).
$$
\end{corollary}

This represents a significant improvement over the simple tail bound
from~\cite[Lem.~7.2]{ALMT:13}.  A similar result
holds for any finite product $C_1 \times \cdots \times C_r$
of closed convex cones.

\begin{proof}
First, recall the numerical inequality
$$
\frac{\econst^{2\zeta} - 2 \zeta - 1}{2}
	\leq \frac{\zeta^2}{1 - 2\abs{\smash{\zeta}}/3}
	\quad\text{for $\abs{\smash{\zeta}} < \tfrac{3}{2}$.}
$$
This estimate allows us to package the two exponential moment bounds from Theorem~\ref{thm:exp-bd} as
$$
\Expect \econst^{\zeta (V_C - \delta(C))}
	\leq \exp\left( \frac{\zeta^2 (\delta(C) \wedge \delta(C^\polar))}{1-2\abs{\smash{\zeta}}/3} \right)
	\quad\text{for $\abs{\smash{\zeta}} < \tfrac{3}{2}$.}
$$
Applying this bound twice, we learn that the exponential moments of the random variable $V_{C_1 \times C_2}$ satisfy
$$
\Expect \econst^{\zeta \, (V_{C_1 \times C_2} - \delta(C_1 \times C_2))}
	= \Expect{} \econst^{ \zeta \, (V_{C_1} - \delta(C_1))} \cdot
	\Expect \econst^{ \zeta \, (V_{C_2} - \delta(C_2))}
	\leq \exp\left( \frac{\zeta^2 \sigma^2}{1-2\abs{\smash{\zeta}}/3} \right).
$$
The first relation follows from the distributional
identity~\eqref{eqn:VC-prod} and the statistical dimension calculation~\eqref{eqn:sdim-prod}.  The Laplace transform method delivers
$$
\Prob\big\{ V_{C_1 \times C_2} - \delta(C_1 \times C_2) \geq \lambda \big\}
	\leq \inf_{\zeta > 0} \left\{ \econst^{- \zeta \lambda} \cdot \exp\left( \frac{\zeta^2 \sigma^2}{1-2\abs{\smash{\zeta}}/3} \right) \right\}
	\leq \exp\left( \frac{-\lambda^2/4}{\sigma^2 + \lambda/3} \right).
$$
We have chosen $\zeta = \lambda/(2\sigma^2 + 2\lambda/3)$ to reach the
second inequality.  We obtain the lower tail bound from the same argument. 
\end{proof}

\section{Examples} \label{sec:examples}

In this section, we demonstrate the vigor of the ideas from Section~\ref{sec:prob-intvols} by applying them to some concrete examples.  The probabilistic viewpoint provides new insights, and it enables us to complete some difficult calculations with minimal effort.

\subsection{The nonnegative orthant}

As a warmup, we begin with an example where it is easy to compute the
intrinsic volumes directly.  The \term{nonnegative orthant} $\R_+^d$
is the polyhedral cone
$$
\R_+^d := \big\{ \vct{x} \in \R^d : x_i \geq 0 \text{ for $i=1, \dots, d$.} \big\}.
$$
The nonnegative orthant is self-dual, which immediately delivers several results.
For typographical felicity, we abbreviate $C = \R_+^d$.  Applying the identity~\eqref{eqn:self-dual} and Theorem~\ref{thm:variance-bd}, we find that
$$
\delta(C) = \Expect[ V_C ] = \half d
\qtq{and}
\Var[V_C] \leq 2 \delta(C) = d.
$$
The tail bound~\eqref{eqn:tail-combined} specializes to
$$
\Prob\big\{ \abs{ V_C - \half d} \geq \lambda \big\}
	\leq 2 \, \exp\left( \frac{-\lambda^2}{2d + 4\lambda/3} \right).
$$
These estimates already provide a significant amount of information about
the intrinsic volumes of the orthant.

How well do these bounds describe the actual behavior of the intrinsic
volumes?  Appealing directly to Definition~\ref{def:intvols},
we can check that $V_C \sim \textsc{binomial}\big(d, \half\big)$.
See, for example,~\cite[Ex.~4.4.7]{Ame:11}.  Therefore,
$$
\Expect[ V_C ] = \half d
\qtq{and}
\Var[ V_C ] = \tfrac{1}{4} d.
$$
Furthermore, the binomial random variable satisfies a sharp tail bound
of the form
$$
\Prob\big\{ \abs{ V_C - \half d } \geq \lambda \big\}
	\leq 2 \, \exp\left( \frac{-2\lambda^2}{d} \right).
$$
We discover that our general results overestimate the variance of
$V_C$ by a factor of four, but they do capture the subgaussian
decay of the intrinsic volumes.

\subsection{The cone of positive-semidefinite matrices}

Our approach to intrinsic volume calculations is most valuable
when there is no explicit formula for the intrinsic volumes
or the expressions are too complicated to evaluate easily.
For a challenge of this type, let us consider the
cone of real positive-semidefinite matrices.  We can
compute the mean and variance of the intrinsic volume
sequence of this cone by combining our methods with
established results from random matrix theory.

The cone $\mathbb{S}_+^n$ consists of
all $n \times n$ positive-semidefinite (psd) matrices:
$$
\mathbb{S}_{+}^n := \big\{ \mtx{X} \in \R_{\rm sym}^{n \times n} : \vct{u}^\transp \mtx{X} \vct{u} \geq 0 \text{ for all $\vct{u} \in \R^n$} \big\}
$$
where $\R_{\rm sym}^{n \times n}$ consists of $n \times n$
symmetric matrices.  This vector space has dimension
$d = n(n+1)/2$.  The psd cone is self-dual with respect
to $\R_{\rm sym}^{n \times n}$, so the expression~\eqref{eqn:self-dual}
shows that the statistical dimension
$$
\delta( \mathbb{S}_+^n ) = \frac{n(n+1)}{4}.
$$
As with the nonnegative orthant, we immediately obtain bounds
on the variance and concentration inequalities for the intrinsic
volumes.

We will use Proposition~\ref{prop:variance}
to compute the variance of the sequence of
intrinsic volumes when $n$ is large.
Let us abbreviate $C = \mathbb{S}_+^n$.
The intrinsic volumes do not depend on the embedding dimension
of the cone, so there is no harm in treating the cone as a
subset of the linear space $\R^{n \times n}$ of square
matrices.  To compute the metric projection of a matrix $\mtx{X} \in \R^{n \times n}$
onto the cone $C$, we first extract the symmetric part of the matrix and then
compute the positive part~\cite[p.~99]{Bha:97} of the Jordan decomposition:
$$
\Proj_C(\mtx{X}) = \Proj_C\big( \half (\mtx{X} + \mtx{X}^\transp) \big)
	= \half (\mtx{X} + \mtx{X}^\transp)_+.
$$
It follows that
$$
\fnorm{\Proj_C(\mtx{X})}^2 = \tfrac{1}{4} \fnorm{ \big(\mtx{X} + \mtx{X}^\transp \big)_+ }^2
	= \tfrac{1}{4} \trace\big[ \big( \mtx{X} + \mtx{X}^\transp \big)_+^2 \big].
$$
Let $\mtx{G}_n \in \R^{n \times n}$ be a matrix with independent standard Gaussian
entries.  Then the matrix $\mtx{W}_n = 2^{-1/2} \big(\mtx{G}_n + \mtx{G}_n^\transp\big) \in \R_{\rm sym}^{n\times n}$ is a member of the Gaussian orthogonal ensemble (GOE).  We have
$$
\fnorm{ \Proj_C(\mtx{G}_n) }^2 = \tfrac{1}{2} \trace \big[ (\mtx{W}_n)_+^2 \big].
$$
To invoke Proposition~\ref{prop:variance}, we must compute the variance of this quantity.

Our method is to renormalize the matrix and invoke asymptotic results for the GOE.  From the formula above,
$$
\Var\big[ \fnorm{ \Proj_C(\mtx{G}_n) }^2 \big]
= \Var\left[ \frac{n}{2} \cdot \trace \big[ \big(n^{-1/2} \mtx{W}_n \big)_+^2 \big] \right]
	= \frac{n^2}{4}\Var\left[ \frac{1}{n} \trace \big[ \big( \mtx{W}_n \big)_+^2 \big] \right].
$$
The final term involves the variance of the  function $h(s) := (s)_+^2 = \max\{s, 0\}^2$ applied to the empirical spectral distribution of \(\mtx W_n\).  In the limit as $n \to \infty$, this variance can be expressed in terms of an integral against a kernel associated with the GOE~\cite[Thm.~9.2]{BaiSil:10}.  
$$
\Var\left[\frac{1}{n} \trace\big[\big( \mtx{W}_n \big)^2_+\big] \right]
	\to \int_{-2}^2 \int_{-2}^2 h'(s) h'(t) \cdot \rho_{\rm GOE}(s,t) \idiff{s}\idiff{t},
$$
where the kernel takes the form
$$
\rho_{\rm GOE}(s,t) = \frac{1}{2 \pi^2} \log\left(
	\frac{4 - st + \sqrt{(4-s^2)(4-t^2)}}{4 - st - \sqrt{(4-s^2)(4-t^2)}} \right).
$$
With the assistance of the Mathematica computer algebra system, we determine that the double integral equals $1 + 16/\pi^2$.  Therefore,
$$
\Var\big[ \fnorm{ \Proj_C(\mtx{G}_n) }^2 \big]
	= \frac{n^2}{4} \left(1 + \frac{16}{\pi^2}\right)
	+ o(n^2)
	\quad\text{as $n \to \infty$.}
$$
Proposition~\ref{prop:variance} yields
$$
\Var[V_C] = \Var\big[ \fnorm{ \Proj_C(\mtx{G}_n) }^2 \big]
	- 2\delta(C)
	= \frac{n^2}{4} \left(\frac{16}{\pi^2} - 1\right) + o(n^2)
	\quad\text{as $n \to \infty$.}
$$
In particular,
$$
\frac{\Var[V_C]}{\delta(C)} \to \frac{16}{\pi^2} - 1
\quad\text{as $n \to \infty$.}
$$
This ratio measures how much the intrinsic volumes are spread
out relative to the size of the cone.

As a point of comparison,
Amelunxen \& B{\"u}rgisser have computed the intrinsic
volumes of the psd cone exactly using methods from differential
geometry~\cite[Thm.~4.1]{AmeBur:14}.  The expressions,
involving Mehta integrals, can be evaluated for low-dimensional
cones, but they have resisted asymptotic analysis.

\subsection{Circular cones}
\label{sec:circular}

A \term{circular cone} $C$ in $\R^d$ with angle $0 \leq \alpha \leq \pi/2$ takes the form
$$
C = \Circ_d(\alpha) := \big\{ \vct{x} \in \R^d : x_1 \geq \enorm{\vct{x}} \cos(\alpha) \big\}.
$$
In particular, this family includes the \term{second-order cone} $\mathbb{L}^{d} := \Circ_{d}(\pi/4)$.  Second-order cones are also known as \term{Lorentz cones}, and they are self-dual.

With some effort, it is possible to work out the intrinsic volumes of a circular cone~\cite[Ex.~4.4.8]{Ame:11}.  Instead, we apply our techniques to compute the mean and variance of the intrinsic volume random variable.  This calculation demonstrates that circular cones with a small angle saturate the variance bound from Theorem~\ref{thm:variance-bd}.  Afterward, we sketch an argument that small circular cones also saturate the upper tail bound from Corollary~\ref{cor:VC-conc}.

\subsubsection{Mean and variance calculations}
\label{sec:circcone-meanvar}

Fix an angle $\alpha \in (0, \pi/2)$.  We consider the circular cone $C = \Circ_d(\alpha)$ where the dimension $d$ is large.  For each unit vector $\vct{u} \in \R^d$, elementary trigonometry shows that
$$
\enormsm{\Proj_C(\vct{u})}^2 = H(\arccos(u_1))
\qtq{where}
H(\beta) := \begin{cases}
	1, & \beta \in [0, \alpha) \\
	\cos^2(\beta - \alpha), & \beta \in [\alpha, \alpha + \pi/2] \\
	0, & \beta \in (\alpha + \pi/2, \pi].
\end{cases}
$$
Recall the polar decomposition $\vct{g} = R \cdot \vct{\theta}$ where $R$ and $\vct{\theta}$ are independent, $R^2$ is a chi-square random variable with $d$ degrees of freedom, and $\vct{\theta}$ is uniformly distributed on the sphere.  With this notation,
\begin{equation} \label{eqn:proj-circ-cone}
\enormsm{\Proj_C(\vct{g})}^2 = R^2 \cdot H\big(\arccos(\theta_1)\big).
\end{equation}
This expression allows us to evaluate the moments of $\enormsm{\Proj_C(\vct{g})}^2$ quickly.

We begin with the statistical dimension.  Combine Proposition~\ref{prop:sdim} with the expression~\eqref{eqn:proj-circ-cone} and integrate in polar coordinates to reach
$$
\delta(C) = \frac{d}{\kappa_d} \int_{0}^\pi H(\beta) \sin^{d-2}(\beta) \idiff{\beta}
\qtq{where}
\kappa_d := \int_{0}^\pi \sin^{d-2}(\beta) \idiff{\beta}.
$$
A more detailed version of this calculation appears in~\cite[Prop.~6.8]{McC:13}.
The function $\beta \mapsto \sin^{d-2}(\beta)$ peaks sharply around $\pi/2$, so
it does little harm to replace $H$ by the function
$\widetilde{H}(\beta) = \cos^2(\beta - \alpha)$ in the integrand.  Computing this simpler integral,
we obtain a closed-form expression plus a remainder term:
\begin{equation} \label{eqn:circcone-sdim}
\delta(C) = d \sin^2(\alpha) + \cos(2\alpha) + \eps_1(\alpha, d).
\end{equation}
We assert that the remainder term is exponentially small as a function of the dimension:
$$
\abs{\eps_1(\alpha, d)} < \sqrt{\frac{\pi}{8}} \cdot d^{3/2} \cdot
\exp\left( - \frac{1}{2} (d-1) \cdot \big(\alpha \wedge (\pi/2-\alpha)\big)^2 \right).
$$
The precise form of the error is not particularly important here, so we omit the details.

The same approach delivers the variance of $V_C$.  Repeating the argument above, we get
$$
\Expect\big[ \enormsm{\Proj_C(\vct{g})}^4 \big]
	= \frac{d(d+2)}{\kappa_d} \int_0^\pi H(\beta)^2 \sin^{d-2}(\beta) \idiff{\beta}.
$$
Replace $H$ with $\widetilde{H}$ and integrate to obtain
\begin{equation} \label{eqn:circcone-second-mom}
\Expect\big[ \enormsm{\Proj_C(\vct{g})}^4 \big]
	= \tfrac{3}{8} d(d+2) - 4(d-2)(d+2) \cos(2\alpha) - (d-4)(d-2)\cos(4\alpha)
	+ \eps_2(\alpha,d).
\end{equation}
The error term satisfies the bound
$$
\abs{\eps_2(\alpha,d)} \leq \sqrt{\frac{\pi}{8}} \cdot d^{3/2} (d+2) \cdot
	\exp\left( - \frac{1}{2} (d-1) \cdot \big(\alpha \wedge (\pi/2-\alpha)\big)^2 \right).
$$
In view of Proposition~\ref{prop:sdim} and Proposition~\ref{prop:variance}, we determine that
\begin{equation} \label{eqn:circcone-var-form}
\Var[V_C] = \Var\big[ \enormsm{\Proj_C(\vct{g})}^2 \big] - 2\delta(C)
	= \Expect\big[ \enormsm{\Proj_C(\vct{g})}^4 \big]
	- \delta(C)^2 - 2\delta(C).
\end{equation}
Combine~\eqref{eqn:circcone-sdim},~\eqref{eqn:circcone-second-mom}, and~\eqref{eqn:circcone-var-form} and simplify to reach
\begin{equation*} \label{eqn:circcone-var}
\Var[V_C]  = \half (d - 2) \sin^2(2\alpha) + \eps_3(\alpha,d).
\end{equation*}
Once again, the remainder term is exponentially small in the dimension $d$ for each fixed $\alpha \in (0, \pi/2)$.

These calculations allow us to compare the variance of $V_C$ with the statistical dimension $\delta(C)$:
$$
\frac{\Var[V_C]}{\delta(C)} = \frac{(d-2) \sin^2(2\alpha)}{2 d \sin^2(\alpha)} + \eps_4(\alpha, d)
	\to 2 \cos^2(\alpha)
	\quad\text{as $d \to \infty$ with fixed $\alpha > 0$.}
$$
By considering a sufficiently small angle $\alpha$, we can find a circular cone $C = \Circ_d(\alpha)$ for which $\Var[V_C]$ is arbitrarily close to $2\delta(C)$.
In conclusion, we cannot improve the constant two in the variance bound from Theorem~\ref{thm:variance-bd}.

\subsubsection{Tail behavior}

Circular cones also exhibit tail behavior that matches the predictions of Corollary~\ref{cor:VC-conc} exactly.  It takes some technical effort to establish this
claim in detail, so we limit ourselves to a sketch of the argument.  These ideas
are drawn from~\cite[Sec.~6.2]{ALMT:13}.

Fix an angle $0 < \alpha \ll \tfrac{\pi}{2}$, and abbreviate $q = \sin^2(\alpha)$.
Consider the circular cone $C = \Circ_d(\alpha)$ where the dimension takes the form $d = 2(n+1)$ where $n$ is a large integer.  In particular, the formula~\eqref{eqn:circcone-sdim} shows that the statistical dimension $\delta(C) \approx d\sin^2(\alpha) \approx 2nq$.  It can be established that the odd intrinsic volumes of $C$ follow a binomial distribution~\cite[Ex.~4.4.8]{Ame:11}:
$$
2 \, v_{2k+1}(C) = \Prob\big\{ Y = k \}
	\qtq{where}
	Y = Y_n \sim \textsc{binomial}(n,q). 
$$
As a consequence of the Gauss--Bonnet Theorem~\cite[Thm.~6.5.5]{SchWei:08}, there is an interlacing result~\cite[Prop.~5.6]{ALMT:13} for the upper tail of $V_C$.
$$
\Prob\big\{ Y \geq k \big\}
	\leq \Prob\big\{ V_C \geq 2k \big\}
	\leq \Prob\big\{ Y \geq k - 1 \big\}.
$$
Thus, accurate probability bounds for $V_C$ follow from bounds for the binomial random variable $Y$.

Our tail inequality, Corollary~\ref{cor:VC-conc}, predicts that $V_C$ has subgaussian behavior for moderate deviations.  To see that circular cones actually display this behavior, we turn to the classical limits for the binomial random variable $Y_n$.  The Laplace--de Moivre central limit theorem states that
$$
\Prob\big\{ Y_n - nq \geq t \sqrt{nq(1-q)} \big\}
	\to 1 - \Phi(t)
	\quad\text{as $n \to \infty$ with $q$ fixed.}
$$
Here, $\Phi$ denotes the distribution function of a standard Gaussian variate.  When $q \approx 0$ and $n$ is large, we can invoke the approximation $\delta(C) \approx 2nq$ and a tail bound for the Gaussian distribution to obtain
$$
\Prob\big\{ V_C - \delta(C) \geq \lambda \sqrt{\delta(C)} \big\}
	\approx \Prob\big\{ V_C - 2nq \geq \lambda \sqrt{2nq(1-q)} \big\}
	\approx \Prob\big\{ Y_n - nq \geq (2^{-1/2} \lambda) \sqrt{nq(1-q)} \big\}
	\approx \econst^{-\lambda^2/4}.
$$
This expression matches the behavior expressed in the weaker tail bound~\eqref{eqn:tail-combined}.  In other words, we see that the intrinsic volumes of a small circular cone have subgaussian concentration for moderate deviations, with variance approximately $2 \delta(C)$.

Corollary~\ref{cor:VC-conc} also predicts that $V_C$ has Poisson tails for very large deviations.  Vanishingly small circular cones display this behavior.  Suppose that $q = q_n = b / n$ for a large constant $b$.
The approximation $\delta(C) \approx 2b$ and Chernoff's bound for the tail of a binomial random variable together give
$$
\Prob\big\{ V_C - \delta(C) \geq \lambda \delta(C) \big\}
	\approx \Prob \big\{ V_C - 2b \geq 2\lambda  b \big\}
	\approx \Prob \big\{ Y_n - b \geq \lambda b \big\}
	\approx \econst^{ b \, (\lambda - (1+\lambda)\log(1+\lambda)) }.
$$
After a change of variables, this formula coincides with the tail bound~\eqref{eqn:exp-bd-1}.  The Chernoff bound is quite accurate in this regime, so
we see that~\eqref{eqn:exp-bd-1} is saturated by vanishingly small circular cones
in high dimensions.

\subsection{Summary of calculations}

We conclude this section with an overview of the statistical dimension and variance
calculations.  See Table~\ref{tab:sdim-var} for this material.
Observe that the ratio of the variance $\Var[V_C]$ 
to the statistical dimension $\delta(C)$
can range from zero to two.  A subspace $L_k$ with dimension $k$
shows that the lower bound is achievable across the entire range of statistical dimensions.
The circular cones $\Circ_d(\alpha)$ show that the upper bound is saturated by cones whose statistical dimension is small.  Amelunxen (personal communication) has conjectured that somewhat tighter bounds are possible when $\delta(C) \approx \half d$, but this remains to be established.

\begin{table}[t]
\caption{\textbf{Statistical dimension and variance calculations.}
This table lists the statistical dimension $\delta(C)$
and the variance $\Var[V_C]$ for several families of convex cones.
The last column displays the limiting ratio $\Var[V_C] /\delta(C)$
as the dimensional parameter tends to infinity.} \label{tab:sdim-var}
\renewcommand{\arraystretch}{1.7}
\begin{tabular}{lcccc}
\toprule
{Cone} & {Notation} & { Mean $\delta(C)$ } & { Variance $\Var[V_C]$ } & { Limit of $\Var[V_C]/\delta(C)$ }\\
\midrule Subspace & $L_k$ & $k$ & 0 & 0 \\
Nonnegative orthant & $\R_+^d$ & $\half d$ & $\tfrac{1}{4} d$ & $\tfrac{1}{2}$ \\
Real positive-semidefinite cone & $\mathbb{S}_+^n$ & $\tfrac{1}{4} n(n+1)$ & $\approx \big(\frac{4}{\pi^2} - \frac{1}{4} \big)  n^2$ & $\tfrac{16}{\pi^2} - 1$ \\
Second-order cone & $\mathbb{L}^{d}$ & $\half d$ & $\approx \half (d - 2)$ & 1 \\
Circular cone & $\Circ_d(\alpha)$ & $\approx d \sin^2(\alpha)$ & $\approx \half (d-2) \sin^2(2\alpha)$ & $2 \cos^2(\alpha)$ \\
\multicolumn{4}{l}{Upper bound (Theorem~\ref{thm:variance-bd})} & 2 \\
\bottomrule
\end{tabular}
\end{table}

\section{Background on conic geometry}
\label{sec:backgr-conv-cones}

This section summarizes the foundational material that we require to establish the master Steiner formula.  We provide sketches or references in lieu of proofs to keep the presentation lean. 
Most of the material here is drawn from the books~\cite{Roc:70,RW:98,Bar:02,SchWei:08}.

\subsection{The tiling induced by a polyhedral cone}
\label{sec:tiling-induced-cone}
This section describes a fundamental decomposition of \(\R^d\) induced by a closed convex cone and its polar.  First, we recall a few basic facts that we will use liberally in our development. Given a cone $C \in \cC_d$, every point $\vct{x} \in \R^d$ can be expressed as an orthogonal sum of the type
\begin{equation} \label{eqn:orth-decomp}
\vct{x} = \Proj_C(\vct{x}) + \Proj_{C^\polar}(\vct{x})
\qtq{where}
\Proj_C(\vct{x}) \perp \Proj_{C^\polar}(\vct{x}).
\end{equation}
We often use the consequence
\begin{equation} \label{eqn:dist-proj}
\enorm{\Proj_{C^\polar}(\vct{x})}^2 = \dist^2( \vct{x}, C ).
\end{equation}
Another outcome is the Pythagorean relation
\begin{equation} \label{eqn:pythag}
\enormsq{ \vct{x} } = \enormsq{ \Proj_C(\vct{x}) } + \enormsq{ \Proj_{C^\polar}(\vct{x}) }.
\end{equation}
See Rockafellar~\cite[pp.~338--341]{Roc:70} for more information about this construction.

Let $C$ be a polyhedral cone.  A \term{face} $F(\vct{u})$ of $C$ with outward normal $\vct{u}$ takes the form
$$
F(\vct{u}) := \big\{ \vct{x} \in C : \ip{ \vct{u} }{ \vct{x} } = \sup\nolimits_{\vct{y} \in C}
	\ip{ \vct{u} }{ \smash{\vct{y}} } \big\}.
$$
The face $F(\vct{u})$ is nonempty if and only if $\vct{u} \in C^\polar$; otherwise, the supremum is infinite. The \term{linear hull} $\lin(K)$ of a convex set $K \subset \R^d$ is the intersection of all subspaces that contain $K$. The \term{dimension} of a face $F$ is the dimension of its linear hull $\lin(F)$.

Recall that the polar of a polyhedral cone is always a polyhedral cone.  The outward normals of a face $F$ of the cone $C$ comprise a face $N_F$ of the polar cone $C^\polar$ called the \term{normal face}:
$$
N_F := \lin(F)^\polar \cap C^\polar.
$$

Each polyhedral cone $C$ induces a tiling of $\R^d$, where each tile is an orthogonal sum of the faces of $C$ and the normal faces of $C^\polar$.  The following statement of this claim amplifies an observation of McMullen~\cite[Lem.~3]{McM:75}. Below, the \term{relative interior} $\relint(K)$ refers to the interior with respect to the relative topology induced by $\R^d$ on the linear hull $\lin(K)$.

\begin{fact}[The tiling induced by a polyhedral cone] \label{fact:tiling}
Let $C \in \cC_d$ be a polyhedral cone.  Then the inverse image of the relative interior of a face $F$ has the orthogonal decomposition
\begin{equation} \label{eqn:face-decomp}
\Proj_C^{-1}\big(\relint(F) \big) = \relint(F) + N_F.
\end{equation}
Moreover, the space $\R^d$ is a disjoint union of the inverse images of the faces
of the cone $C$:
\begin{equation} \label{eqn:Rd-partition}
\R^d = \bigsqcup_{\text{$F$ a face of $C$}} \big( \relint(F) + N_F \big).
\end{equation}
\end{fact}

\noindent
Fact~\ref{fact:tiling} is almost obvious from the orthogonal decomposition~\eqref{eqn:orth-decomp}.  See~\cite[Prop.~A.8]{McC:13}
for a detailed proof.

\subsection{The solid angle of a cone}
\label{sec:polytope-angles}

Let $C$ be a convex cone whose linear hull is $j$-dimensional.
The \term{solid angle} of the cone is defined as
\begin{equation} \label{eq:solid-angle-probab}
\sang(C) := \frac{1}{(2\pi)^{j/2}} \int_C \econst^{-\enormsm{\vct{x}}^2/2} \idiff{\vct{x}}
	= \Prob\big\{ \vct{g}_C \in C \big\} = \Prob\big\{ \vct{\theta}_C \in C \big\}.
\end{equation}
The volume element $\diff{\vct{x}}$ derives from the Lebesgue measure
on the linear hull $\lin(C)$.  The random vector $\vct{g}_C$ has the standard
Gaussian distribution on $\lin(C)$, and $\vct{\theta}_C$ is uniformly distributed on
the unit sphere in $\lin(C)$.  We use the convention that the unit sphere
in the zero-dimensional Euclidean space $\R^0$ is the set $\sphere{-1} := \{0\}$.

Let $C$ be a polyhedral cone, and let $F$ be a face of $C$ with normal face $N_F$.
The \term{internal angle} of $F$ is the solid angle $\sang(F)$,
while the \term{external angle} of $F$ is the solid angle $\sang(N_F)$.
The intrinsic volumes of a polyhedral cone can be written in terms
of the internal and external angles of the faces.

\begin{fact}[Intrinsic volumes and polyhedral angles] \label{fact:polytope-angles}
Let $C \in \cC_d$ be a polyhedral cone, and let $\coll{F}_k(C)$ be the family of $k$-dimensional faces of $C$.  Then
\begin{equation*} \label{eq:poly-angle-int-vols}
	v_k(C) = \sum\nolimits_{F \in \coll{F}_k(C)}
	\sang(F) \sang(N_F).
\end{equation*}
\end{fact}

\noindent
Fact~\ref{fact:polytope-angles} is a direct consequence of
the Definition~\ref{def:intvols} of the intrinsic volumes of a polyhedral cone,
the orthogonal decomposition~\eqref{eqn:face-decomp} of the inverse image of a face, and the geometric interpretation~\eqref{eq:solid-angle-probab} of the solid angles.
This result can be traced at least as far back as
~\cite{McM:75}; see also~\cite[Eqn.~(6.47)]{SchWei:08}.
A complete proof appears in~\cite[Prop.~A.8]{McC:13}.

\begin{remark}[Alternative notation]
In the literature, the internal angle of a face $F$ of a cone $C$ is often denoted by $\beta(\vct{0}, F)$, and the external angle is often denoted by $\gamma(F,C)$.
\end{remark}

\subsection{The Hausdorff topology on convex cones}
\label{sec:relat-betw-conic}

In this section, we develop a metric topology on the class $\cC_d$ of closed convex cones.  This topology leads to notions of approximation and convergence, and it provides a way to extend results for polyhedral cones to general closed convex cones.  See~\cite[Sec.~3.2]{Ame:11} for a more comprehensive treatment.

To construct an appropriate metric, we begin by defining the angular distance between two nonzero vectors:
$$
\dist_s(\vct{x}, \vct{y}) := \arccos\left(
	\frac{\ip{\vct{x}}{\smash{\vct{y}}}}{\enormsm{\vct{x}} \enormsm{\vct{y}}} \right)
	\quad\text{for}\quad \vct{x}, \vct{y} \in \R^d \setminus \{\vct{0} \}.
$$
We instate the conventions that $\dist_s(\vct{0}, \vct{0}) = 0$ and $\dist_s(\vct{x}, \vct{0}) = \dist_s(\vct{0}, \vct{x}) = \pi/2$ for $\vct{x} \neq \vct{0}$.
This definition extends to closed convex cones $C, C' \in \cC_d$ via the rule
$$
\dist_s(C, C') :=
	\inf\nolimits_{\substack{\vct{x} \in C \\ \vct{y}\in C'}} \dist_s(\vct{x}, \vct{y})
	\quad\text{when $C, C' \neq \{ \vct{0} \}$.}
$$
The trivial cone $\{\vct{0}\}$ demands special attention.  We set $\dist_s(\{\vct{0}\}, \{\vct{0}\}) = 0$, while $\dist_s(\{\vct{0}\}, C) = \dist_s(C, \{\vct{0}\}) = \pi/2$
when $C \neq \{\vct{0}\}$.

The \term{angular expansion} $\aTube(C, \alpha)$ of a cone $C \in \cC_d$ by an angle $0 \leq \alpha \leq 2\pi$ is the union of all rays that lie within an angle $\alpha$ of the cone.  Equivalently,
$$
\aTube(C, \alpha) := \big\{ \vct{x} \in \R^d : \dist_s(\vct{x}, \vct{y}) \leq \alpha
	\text{ for some $\vct{y} \in C$} \big\}.
$$
Note that the expansion $\aTube(C,\alpha)$ of a convex cone need not be convex for any $\alpha > 0$.  For instance, the angular expansion of a proper subspace is never convex.

Define the \term{conic Hausdorff metric} $\distH(C_1,C_2)$ between two cones $C_1, C_2 \in \cC_d$ by
$$
\distH(C_1,C_2) := \inf\big\{\alpha \geq 0 : \aTube(C_1, \alpha) \supset C_2
\text{ and } \aTube(C_2, \alpha) \supset C_1 \big\}
\quad\text{for $C_1,C_2 \in \cC_d$.}
$$
We equip $\cC_d$ with the conic Hausdorff metric and the associated metric topology to form a compact metric space.  It is not hard to check~\cite[Prop.~3.2.4]{Ame:11} that polarity is a local isometry on $\cC_d$:
\begin{equation} \label{eqn:polar-isom}
\text{For $\alpha < \pi/2$,} \quad
\distH(C_1,C_2) = \alpha
\qtq{implies}
\dist(C_1^\polar, C_2^\polar) = \alpha.
\end{equation}
When we write expressions like $C_i \to C$ for closed convex cones, we are always referring to convergence in the conic Hausdorff metric.  The property~\eqref{eqn:polar-isom} ensures that $C_i \to C$ if and only if $C_i^\polar \to C^\polar$.

A basic principle in the analysis of metric spaces is to identify a dense subset that consists of points with additional regularity.  We can make arguments that exploit this regularity and apply a limiting procedure to extend the claim to the rest of the space.
To that end, let us demonstrate that the polyhedral cones form a dense subset of $\cC_d$.
The approach mirrors~\cite[Thm.~1.8.13]{Sch:93}.

\begin{fact}[Polyhedral cones are dense] \label{fact:poly-dense}
Let $C \in \cC_d$ be a closed convex cone.  For each $\eps > 0$, there is a polyhedral cone $C_{\eps} \in \cC_d$ that satisfies $\distH(C, C_{\eps}) < \eps$.
\end{fact}

\begin{proof}[Proof sketch]
We may assume $C \neq \{\vct{0}\}$.  Let $\coll{X}$ be a finite $\eps$-cover of the set $C \cap \sphere{d-1}$ with respect to the angular distance.  That is,
$$
\coll{X} = \big\{ \vct{x}_i : i = 1, \dots, N_{\eps} \big\} \subset C \cap \sphere{d-1}
\qtq{and}
\min\nolimits_i \dist_s(\vct{x}, \vct{x}_i) < \eps
\qtq{for all $\vct{x} \in C \cap \sphere{d-1}$.}
$$
Consider the convex cone $C_\eps$ generated by $\coll{X}$:
$$
C_{\eps} := \cone(\coll{X}) :=
\left\{ \sum\nolimits_{i=1}^{N_{\eps}} \tau_i \vct{x}_i : \tau_i \geq 0 \right\}.
$$
The cone $C_{\eps}$ is polyhedral, and it satisfies $\distH(C, C_{\eps}) < \eps$.
\end{proof}

In order to perform limiting arguments, it helps to work with continuous functions.  The next result ensures that projection onto a cone is continuous with respect to the conic Hausdorff metric.

\begin{fact}[Continuity of the projection] \label{fact:proj-cont}
Consider a sequence $( C_i )_{i \in\mathbb{N}}$ of closed convex cones in $\cC_d$ where $C_i \to C$ in the conic Hausdorff metric.  For each $\vct{x} \in \R^d$, the projection $\Proj_{C_i}(\vct{x}) \to \Proj_C(\vct{x})$ as $i \to \infty$.
\end{fact}

\noindent
The proof is a straightforward exercise in elementary analysis,
so we refer the reader to~\cite[Prop.~3.8]{McC:13} for details.
This result has a Euclidean analog~\cite[Lem.~1.8.9]{Sch:93}.

\section{Proof of the master Steiner formula}
\label{sec:proof-gener-steiner}

This section contains the proof of
Theorem~\ref{thm:general-steiner}.
In Section~\ref{sec:steiner-poly}, we establish a restricted version
of the master Steiner formula for polyhedral cones.
In Section~\ref{sec:continuity-intvols}, we apply this
basic result to prove that the intrinsic volumes of a closed
convex cone are well defined, and we verify that the
intrinsic volumes are continuous with respect to the conic Hausdorff metric.
Afterward, we use an approximation argument to extend the
master Steiner formula to closed convex cones in Section~\ref{sec:extend-cones},
and we remove the restrictions on the function $f$ in Section~\ref{sec:extend-functions}.

\subsection{Polyhedral cones and bounded continuous functions}
\label{sec:steiner-poly}

We begin with a specialized version of the master Steiner formula
that restricts the cone $C$ to be polyhedral and the function $f$
to be bounded and continuous.  This argument
contains all the essential geometric ideas.

\begin{lemma}[Master Steiner formula for polyhedral cones]
\label{lem:general-steiner-poly}
Let $f : \R_+^2 \to \R$ be a bounded continuous function, and let $C \in \cC_d$ be a polyhedral cone.
Then the geometric functional $\phi_f$ defined in~\eqref{eq:phi-f}
admits the expression
\begin{equation} \label{eq:gen-steiner-pf}
	\phi_f(C) = \sum_{k=0}^d \phi_f(L_k) \cdot v_k(C)
\end{equation}
where $L_k$ is a $k$-dimensional subspace of $\R^d$ and
the conic intrinsic volumes $v_k$ are introduced in
Definition~\ref{def:intvols}.
\end{lemma}

\begin{proof}
Define the random variables $\vct{u} = \Proj_C(\vct{g})$
and $\vct{w} = \Proj_{C^\polar}(\vct{g})$.
The tiling~\eqref{eqn:Rd-partition} induced by a polyhedral cone
allows us to decompose the functional $\phi_f$ in terms of the
faces of the cone $C$.
\begin{equation} \label{eqn:phi-f-faces}
\phi_f(C) = \Expect\big[ f\big( \enormsm{\vct{u}}^2, 
	\enormsm{\vct{w}}^2 \big) \big]
	= \sum_{k=0}^d \ \sum_{F \in \coll{F}_k(C)}
	\Expect \big[ f\big(\enormsm{\vct{u}}^2,
	\enormsm{\vct{w}}^2 \big) \cdot
	\pInd_{\relint(F)}(\vct{u}) \big]
\end{equation}
where $\coll{F}_k(C)$ is the set of $k$-dimensional faces of $C$
and $\pInd_{A}$ is the 0--1 indicator function of a Borel
set $A$.

We need to find an alternative expression for the expectation
remaining in~\eqref{eqn:phi-f-faces}.  Fix a $k$-dimensional
face $F$ of $C$ with normal face $N_F$.  The orthogonal
decomposition~\eqref{eqn:face-decomp}
of the inverse image $\Proj_C^{-1}\big(\relint(F)\big)$
implies that we can integrate over $F$ and $N_F$ independently.
$$
\Expect\big[ f\big( \enormsm{\vct{u}}^2, \enormsm{\vct{w}}^2  \big) \cdot
	\pInd_{\relint(F)}(\vct{u}) \big]
= \frac{1}{(2\pi)^{d/2}} \int_{\relint(F)} \diff{\vct{x}} 	\int_{N_F} \diff{\vct{y}} \cdot
	f\big( \enormsm{\vct{x}}^2, \enormsm{\vct{y}}^2 \big) \cdot
	\econst^{-(\enormsm{\vct{x}}^2 + \enormsm{\vct{y}}^2)/2}.
$$
This identity relies on the Pythagorean relation~\eqref{eqn:pythag}.
The volume elements $\diff{\vct{x}}$ and $\diff{\vct{y}}$ derive from
the Lebesgue measures on $\lin(F)$ and $\lin(N_F)$.
Some care is required for the face $F = \{\vct{0}\}$,
in which case $\diff{\vct{x}}$ is the Dirac measure
at the origin; a similar issue arises when
$N_F = \{ \vct{0} \}$.

To continue, we convert each of the integrals to
polar coordinates~\cite[Thm.~2.49]{Fol:99}.  This
step gives
\begin{equation} \label{eqn:post-polar}
\Expect\big[ f\big( \enormsm{\vct{u}}^2, \enormsm{\vct{w}}^2  \big) \cdot
	\pInd_{\relint(F)}(\vct{u}) \big]
	= \left( \int_{\relint(F) \cap \sphere{k-1}} \diff{\bar{\sigma}_{k-1}} \right)
	\left( \int_{N_F \cap \sphere{d-k-1}} \diff{\bar{\sigma}_{d-k-1}} \right)
	\cdot I_f(k,d)
\end{equation}
where $\bar{\sigma}_{j-1}$ denotes the uniform measure on the sphere
$\sphere{j-1}$.  The quantity $I_f(k,d)$ depends only
on the function $f$ and the two indices $k$ and $d$:
$$
I_f(k,d) := \frac{\sigma_{k-1}\big(\sphere{k-1}\big) \cdot \sigma_{d-k-1}\big(\sphere{d-k-1}\big)}{(2\pi)^{d/2}}
	\int_0^\infty \int_0^\infty f\big(s^2,t^2 \big) \cdot s^{k-1} t^{d-k-1}
	\econst^{-(s^2+t^2)/2} \idiff{s} \idiff{t}
	\quad\text{when $1 \leq k  \leq d - 1$}
$$
and
$$
I_f(0,d) := \frac{\sigma_{d-1}(\sphere{d-1})}{(2\pi)^{d/2}}
	\int_0^\infty f\big(0, t^2\big) \cdot t^{d-1} \econst^{-t^2/2} \idiff{t}
\qtq{and}
I_f(d,d) := \frac{\sigma_{d-1}(\sphere{d-1})}{(2\pi)^{d/2}}
	\int_0^\infty f\big(s^2, 0\big) \cdot s^{d-1} \econst^{-s^2/2} \idiff{s}.
$$
We do not need these formulas for $I_f$, but we have included them for reference.

In view of the identity~\eqref{eq:solid-angle-probab}
for the solid angle of a cone, the expression~\eqref{eqn:post-polar} implies that
\begin{equation} \label{eqn:face-expect}
\Expect\big[ f\big( \enormsm{\vct{u}}^2, \enormsm{\vct{w}}^2  \big) \cdot
	\pInd_{\relint(F)}(\vct{u}) \big]
	= \sang(F) \sang(N_F)
	\cdot I_f(k,d).
\end{equation}
We have employed the fact that the solid angle of a cone coincides with the
solid angle of its relative interior.
The geometry of the face $F$ only enters this expression through the
presence of the solid angles.

We are almost done now.
Combine the decomposition~\eqref{eqn:phi-f-faces} and the identity~\eqref{eqn:face-expect}
to reach
\begin{equation} \label{eqn:phi-f-almost}
\phi_f(C)
	= \sum_{k=0}^d I_f(k,d) \cdot \left( \sum\nolimits_{F \in \coll{F}_k(C)} \sang(F) \sang(N_F) \right)
	 = \sum_{k=0}^d I_f(k,d) \cdot v_k(C).
\end{equation}
The second relation follows from Fact~\ref{fact:polytope-angles}, which
expresses the intrinsic volumes in terms of the internal and external
angles of the cone $C$.
Finally, we must identify an alternative representation
for the coefficients $I_f(k,d)$.  Recall that a $j$-dimensional subspace
$L_j$ of $\R^d$ is a polyhedral cone with $v_j(L_j) = 1$ and $v_k(L_j) = 0$
for $k \neq j$.
Applying the formula~\eqref{eqn:phi-f-almost}
to the subspace $L_j$, we learn that
$$
\phi_f(L_j) = I_f(j,d)
\quad\text{for $j = 0, 1, 2, \dots, d$.}
$$
Substitute these identities into~\eqref{eqn:phi-f-almost} to complete the
proof of~\eqref{eq:gen-steiner-pf}.
\end{proof}

\subsection{Continuity of intrinsic volumes}
\label{sec:continuity-intvols}

To carry out our plan, we need to verify that the conic intrinsic
volumes of $C$ are well defined and continuous with respect to the conic Hausdorff metric.

\begin{proposition}[Intrinsic volumes of convex cones]\label{prop:intv-cont}
Consider a closed convex cone $C \in \cC_d$.
\begin{enumerate} \setlength{\itemsep}{2mm}
\item	{\textbf{Well-definition.}}  There is a sequence $(C_i)_{i \in \mathbb{N}}$ of polyhedral cones in $\cC_d$ that converges to $C$ in the conic Hausdorff metric.
For each index $k$, the limit $\lim_{i \to \infty} v_k(C_i)$ exists,
and it is independent of the sequence of polyhedral cones.  Therefore, we may define
\begin{equation} \label{eqn:intvol-def-pf}
v_k(C) := \lim_{i \to \infty} v_k(C_i)
\quad\text{for $k = 0, 1, 2, \dots, d$.}
\end{equation}

\item	{\textbf{Continuity.}}  Let $(C_i)_{i\in\mathbb{N}}$ be any sequence of cones in $\cC_d$ that converges to $C$ in the conic Hausdorff metric.  Then
$$
\lim_{i \to \infty} v_k(C_i) = v_k(C)
\quad\text{for $k = 0, 1, 2, \dots, d$.}
$$
\end{enumerate}
\end{proposition}

Proposition~\ref{prop:intv-cont} is not new.  For instance,
it is an immediate consequence of the corresponding fact~\cite[Thm.~6.5.2(b)]{SchWei:08}
about spherical intrinsic volumes.  Here,
we develop the result as a consequence of Lemma~\ref{lem:general-steiner-poly}
and the continuity of the projection map, Fact~\ref{fact:proj-cont}.
We believe that this argument provides an attractive alternative
to the standard methods.  Our approach rests on the following lemma.

\begin{lemma} \label{lem:chisq-fns}
Let $X_k$ denote a chi-square random variable
with $k$ degrees of freedom.  For each $d \in \mathbb{N}$,
there is a family  $\big\{f_1, f_2, f_3, \dots, f_d \big\}$
of bounded continuous functions on $\R_+$ with the property that
$$
\Expect\big[ f_j(X_k) \big] = \begin{cases} 1, & j = k \\ 0, & j \neq k. \end{cases}
$$
\end{lemma}

\begin{proof}
For each $k = 1, 2, 3, \dots, d$, consider the function $\rho_k : \R_+ \to \R$ defined by
$$
\rho_k(s) = \frac{1}{2^{k/2} \Gamma(k/2)} s^{k/2} \econst^{-s/2}
	\quad\text{for $s \geq 0$.}
$$
These functions are bounded and continuous, and they compose a linearly independent family~\cite[Chap.~5]{CL00:Course-Approximation}.  Introduce the $d$-dimensional linear space
$P := \lin\big\{ \rho_1, \dots, \rho_d \big\}$ equipped with the inner product
$$
\ip{ \smash{f} }{ \smash{\rho} } := \int_{0}^\infty f(s) \rho(s) \frac{\diff{s}}{s}
\quad\text{for $f, \rho \in P$.}
$$
Standard arguments~\cite[Lem.~8.6-2]{Kre89:Introductory-Functional}
show that $\big\{ \rho_1, \dots, \rho_d \big\}$ induces a biorthogonal
system $\big\{ f_1, \dots, f_d \big\} \subset P$.
By construction, the functions $f_j$ identify the number of degrees of freedom in a chi-square random variable.  Indeed, let $X_k$ follow the chi-square distribution with $k$ degrees of freedom.  Then
$$
\Expect\big[ f_j(X_k) \big]
	= \int_0^\infty f_j(s) \cdot
	\frac{1}{2^{k/2} \Gamma(k/2)} s^{k/2 - 1} \econst^{-s/2}\idiff{s}
	= \ip{ \smash{f_j} }{ \smash{\rho_k} }
	= \begin{cases} 1, & j = k \\ 0, & j \neq k. \end{cases}
$$
This is the advertised result.
\end{proof}

We can use the functions from Lemma~\ref{lem:chisq-fns} in combination with Lemma~\ref{lem:general-steiner-poly} to isolate the properties of individual intrinsic volumes.

\begin{proof}[Proof of Proposition~\ref{prop:intv-cont}]
Let $C \in \cC_d$ be a closed convex cone.
Fact~\ref{fact:poly-dense} implies that there is a sequence
$(C_i)_{i \in \mathbb{N}}$ of polyhedral cones in $\cC_d$
for which $C_i \to C$.  As a consequence, $C_i^\polar \to C^\polar$ as well.

Consider the family $\big\{ f_1, \dots, f_d \big\}$ of functions
promised by Lemma~\ref{lem:chisq-fns}.
For each index $j \geq 1$, Lemma~\ref{lem:general-steiner-poly}
shows that
$$
\Expect \big[ f_j\big( \enormsm{\Proj_{C_i}(\vct{g})}^2 \big) \big]
	= \sum_{k=0}^d \Expect\big[ f_j(X_k) \big] \cdot v_k(C_i)
	= v_j(C_i)
	\quad\text{for $i \in \mathbb{N}$.}
$$
We claim that
\begin{equation}\label{eqn:vj-lim}
v_j(C_i) = \Expect \big[ f_j\big( \enormsm{\Proj_{C_i}(\vct{g})}^2 \big) \big]
	\to \Expect\big[ f_j\big(\enormsm{\Proj_C(\vct{g})}^2 \big) \big]
	\quad\text{as $i \to \infty$.}
\end{equation}
The limit in~\eqref{eqn:vj-lim} does not depend on the choice of sequence,
so the definition~\eqref{eqn:intvol-def-pf} of the intrinsic
volumes of $C$ is valid for each index $k \geq 1$.
For the intrinsic volume $v_0$, we simply note that
$$
v_0(C_i)
	= v_d( C_i^\polar ) \to v_d( C^\polar )
	\quad\text{as $i \to \infty$}
$$
as a consequence of the definition of $v_d$.  This limit is unambiguous, so $v_0$ is also well-defined.

To justify the calculation in~\eqref{eqn:vj-lim},
we apply the dominated convergence theorem to pass
the limit through the expectation.  Fact~\ref{fact:proj-cont}
shows that the metric projection is continuous; the Euclidean norm
and the functions $f_j$ are also continuous.  Thus, we have the pointwise limit
$$
f_j\big(\enormsm{\Proj_{C_i}(\vct{x})}^2 \big) \to
	f_j\big(\enormsm{\Proj_C(\vct{x})}^2 \big)
	\quad\text{as $i \to \infty$ for $\vct{x} \in \R^d$.}
$$
The integrands are controlled by an integrable function because
$f_j$ is bounded:
$$
 \abs{ f_j\big( \enormsm{\Proj_{C_i}(\vct{g})}^2 \big) } 
	\leq \sup_{s \geq 0 } \abs{f_j(s)}
	\quad\text{for $i \in \mathbb{N}$.}
$$
Dominated convergence applies, which ensures that~\eqref{eqn:vj-lim} is correct.

From here, it is easy to verify the continuity of intrinsic volumes.
Suppose that $(C_i)_{ i \in \mathbb{N} }$ is a sequence
of closed convex cones in $\cC_d$ for which $C_i \to C$.
For each index $k \geq 0$,
we can find a sequence $(C_i')_{ i \in \mathbb{N} }$
of polyhedral cones in $\cC_d$ for which
$$
\distH(C_i', C_i) < i^{-1}
\qtq{and}
\abs{ v_k(C_i') - v_k(C_i) } < i^{-1}
\quad\text{for $i \in \mathbb{N}$.}
$$
This point follows from the density of polyhedral cones
in $\cC_d$ stated in Fact~\ref{fact:poly-dense} and the
definition~\eqref{eqn:intvol-def-pf} of the intrinsic volumes.
Since $C_i \to C$, this construction ensures that $C_i' \to C$.
By definition of the intrinsic volumes, $v_k(C_i') \to v_k(C)$.
But then we must conclude that $v_k(C_i) \to v_k(C)$.
\end{proof}

\subsection{Extension to general convex cones}
\label{sec:extend-cones}

Next, let us extend the master Steiner formula, Lemma~\ref{lem:general-steiner-poly}
from polyhedral cones to closed convex cones.  Our strategy is to approximate a convex cone with a sequence of polyhedral cones, apply Lemma~\ref{lem:general-steiner-poly} to each member of the sequence, and use continuity to take the limit.

\begin{lemma}[Extension to closed convex cones] \label{lem:steiner-cone}
Let $f : \R_+^2 \to \R$ be a bounded continuous function,
and let $C \in \cC_d$ be a closed convex cone.  Then the master
Steiner formula~\eqref{eq:gen-steiner-pf} still holds.
\end{lemma}

\begin{proof} Fact~\ref{fact:poly-dense}
ensures that polyhedral cones form a dense subset of $\cC_d$,
so there is a sequence $( C_i )_{ i \in \mathbb{N} }$ of
polyhedral cones in $\cC_d$ for which $C_i \to C$ as $i \to \infty$.
We also have the limit $C_i^\polar \to C^\polar$.
Lemma~\ref{lem:general-steiner-poly} implies that
$$
\Expect\big[ f\big( \enormsm{\Proj_{C_i}(\vct{g})}^2, \
	\enormsm{\Proj_{C_i^\polar}(\vct{g})}^2 \big) \big]
	= \sum_{k=0}^d \phi_f(L_k) \cdot v_k(C_i)
	\quad\text{for $i \in \mathbb{N}$.}
$$
Taking the limit as $i \to \infty$, we reach
$$
\Expect\big[ f\big( \enormsm{\Proj_{C}(\vct{g})}^2, \
	\enormsm{\Proj_{{C}^\polar}(\vct{g})}^2 \big) \big]
	 = \sum_{k=0}^d \phi_f(L_k) \cdot v_k(C).
$$
To justify the limit on the left-hand side, we invoke
the dominated convergence theorem.  This act is legal
because $f$ is bounded and continuous, the squared
Euclidean norm is continuous, and the metric projector
is continuous.  The limit on the right-hand side
follows from the continuity of intrinsic volumes expressed in
Proposition~\ref{prop:intv-cont}.
\end{proof}

\subsection{Extension to integrable functions}
\label{sec:extend-functions}

We are now prepared to complete the proof of the master Steiner formula
that we announced in Section~\ref{sec:master-steiner}.
All that remains is to expand the class of functions that we can consider.
The following lemma contains the outstanding claims of
Theorem~\ref{thm:general-steiner}.

\begin{lemma}[Extension to integrable functions] \label{lem:steiner-poly-extend}
Let $f : \R_+^2 \to \R$ be a Borel measurable function,
and let $C \in \cC_d$ be a closed convex cone.  Then the master
Steiner formula~\eqref{eq:gen-steiner-pf} still holds,
provided that each expectation is finite.
\end{lemma}

\begin{proof}
Let us reinterpret Lemma~\ref{lem:steiner-cone} as a statement
about measures.  The Banach space $C_0(\R_+^2)$ consists of
bounded and continuous real-valued functions on $\R_+^2$ that tend to zero
at infinity.
Consider a function $h \in C_0(\R_+^2)$, and
observe that the left-hand side of~\eqref{eq:gen-steiner-pf} can be written as
$$
\phi_h(C) = \Expect\big[ h\big( \enormsm{\Proj_C(\vct{g})}^2, \ 
	\enormsm{\Proj_{C^\polar}(\vct{g})}^2 \big) \big]
	= \int h(s,t) \idiff{\mu}(s,t)
$$
where the measure $\mu$ is defined for each Borel set $A \subset \R_+^2$ by the
rule
$$
\mu(A) := \Prob\big\{ \big(\enormsm{\Proj_C(\vct{g})}^2, \
	\enormsm{\Proj_{C^\polar}(\vct{g})}^2 \big) \in A \big\}.
$$
Similarly, the right-hand side of~\eqref{eq:gen-steiner-pf} can be written
as
$$
\sum_{k=0}^d \phi_f(L_k) \cdot v_k(C)
	= \sum_{k=0}^d \left( \int h(s,t) \idiff{\mu}_k(s,t) \right) \cdot v_k(C)
$$
where the measure $\mu_k$ is defined via
$$
\mu_k(A) := \Prob\big\{ \big(\enormsm{\Proj_{L_k}(\vct{g})}^2, \	
	\enormsm{\Proj_{{L_k}^\polar}(\vct{g})}^2 \big) \in A \big\}.
$$
As a consequence, Lemma~\ref{lem:steiner-cone} demonstrates that
\begin{equation} \label{eqn:test-fns}
\int h \idiff{\mu} = \int h \idiff{\left(\sum_{k=0}^d v_k(C) \mu_k\right)}
\quad\text{for $h \in C_0(\R_+^2)$.}
\end{equation}
We claim that~\eqref{eqn:test-fns} guarantees the equality of measures
\begin{equation} \label{eqn:measure-equal}
\mu = \sum_{k=0}^d v_k(C) \cdot \mu_k.
\end{equation}
Because the measures are equal, it holds for each nonnegative Borel measurable
function $f_+ : \R_+^2 \to \R_+$ that
$$
\int f_+ \idiff{\mu} = \sum_{k=0}^d \left( \int f_+ \idiff{\mu_k} \right) \cdot v_k(C).
$$
We can replace $f_+$ with any Borel measurable function $f: \R_+^2 \to \R$,
provided that all the integrals remain finite.  Reinterpreted, this observation yields
the conclusion.

Finally, we justify the claim~\eqref{eqn:measure-equal}.
The dual of $C_0(\R_+^2)$ can be identified
as the Banach space $\mathbb{M}(\R_+^2)$ of regular Borel measures,
acting on functions by integration~\cite[Thm.~6.19]{Rud87:Real-Complex}.
Therefore, $C_0(\R_+^2)$ separates points in
$\mathbb{M}(\R_+^2)$~\cite[Sec.~3.14]{Rud91:Functional-Analysis}.
Each of the measures $\mu$ and $\mu_k$ is the push-forward of
the standard Gaussian measure $\gamma_d$ by a continuous function,
so each one is a regular Borel probability
measure~\cite[pp.~174, 185]{Bil95:Probability-Measure}.
Therefore, the collection of integrals in~\eqref{eqn:test-fns}
guarantees the equality of measures in~\eqref{eqn:measure-equal}. 
\end{proof}

\section*{Acknowledgments}

The authors thank Dennis Amelunxen and Martin Lotz for inspiring conversations and for their  thoughtful comments on this material.  This research was supported by ONR awards N00014-08-1-0883 and N00014-11-1002, AFOSR award FA9550-09-1-0643, and a Sloan Research Fellowship.

\bibliographystyle{alpha}

\end{document}